\documentclass{amsart}
\usepackage[all,ps,cmtip]{xy}
\usepackage{amsmath, amssymb}
\usepackage{amscd}

\newtheorem{theorem}{Theorem}[section]
\newtheorem{lemma}[theorem]{Lemma}

\newtheorem{corollary}[theorem]{Corollary}
\newtheorem{claim}[theorem]{Claim}
\newtheorem{proposition}[theorem]{Proposition}

\theoremstyle{definition}
\newtheorem{definition}[theorem]{Definition}
\newtheorem{example}[theorem]{Example}

\newtheorem{question}[theorem]{Question}

\theoremstyle{remark}
\newtheorem{remark}[theorem]{Remark}

\numberwithin{equation}{section}

\DeclareMathOperator{\rank}{rk}

\DeclareMathOperator{\Sym}{S}

\DeclareMathOperator{\Td}{Td}

\DeclareMathOperator{\codim}{codim}

\DeclareMathOperator{\Hom}{Hom}

\DeclareMathOperator{\Coh}{Coh}

\DeclareMathOperator{\D}{D}

\DeclareMathOperator{\Spec}{Spec}

\DeclareMathOperator{\ch}{ch}

\DeclareMathOperator{\pr}{pr}

\DeclareMathOperator{\chr}{char}

\DeclareMathOperator{\reg}{reg}

\DeclareMathOperator{\Gal}{Gal}

\DeclareMathOperator{\Fil}{Fil}

\DeclareMathOperator{\T}{T}

\DeclareMathOperator{\gr}{gr}

\makeatletter
\@namedef{subjclassname@2020}{\textup{2020}
Mathematics Subject Classification}
\makeatother

\begin{document}

\title{Frobenius type positivity of Hodge bundles and applications}

\author{Hao Max Sun}

\address{Department of Mathematics, Shanghai Normal University, Shanghai 200234, People's Republic of China}

\email{hsun@shnu.edu.cn, hsunmath@gmail.com}




\date{December 9, 2025}


\dedicatory {In memory of Professor Eckart Viehweg}

\begin{abstract}
We define two new Frobenius type positivity notions, $F^t$-ampleness and $F^tGG$-ampleness, for coherent sheaves, show that they are stronger than ordinary ampleness and establish some of their basic properties such as positivity of top Chern characters and Kawamata-Viehweg type vanishing theorem. The analogous Frobenius type semipositivity notions and their basic properties have also been established.

By Illusie's decomposition of the relative de Rham complexes, we prove the $F^1GG$-semipositivity of Hodge bundles for semistable morphisms of varieties. This strengthens and generalizes various previous semipositivity results of Fujita, Kawamata, Koll\'ar, Viehweg, Fujino and others. As a consequence we obtain the positivity of top Chern characters of Hodge bundles and Kawamata-Viehweg-Koll\'ar type vanishing theorems for some semistable morphisms. Our results are also applied to study the images of Fano varieties under semistable morphisms and the slope inequality for a family of semistable curves.
\end{abstract}

\maketitle

\setcounter{tocdepth}{1}
\tableofcontents

\section{Introduction}
Let $f:X \rightarrow Y$ be a surjective morphism between complex smooth projective varieties. The study of the (semi)positivity of the Hodge bundles $R^if_*\omega_{X/Y}$ and $f_*\omega^{\otimes m}_{X/Y}$ originates from different areas such as birational geometry, moduli theory and Hodge theory.

By the theory of variations of Hodge structure, Griffiths \cite{Gri} firstly showed that $f_*\omega_{X/Y}$ is a semipositive locally free sheaf in the sense of Griffiths when $f$ is smooth. Moreover, Berndtsson \cite{Ber} (see also \cite{Schu}) proved that $f_*\omega_{X/Y}$ is semipositive in the sense of Nakano by $L^2$ method. Without the smoothness assumption of $f$, Fujita \cite{Fujita} proved the nefness of $f_*\omega_{X/Y}$ when $Y$ is a curve. This has been generalized by Kawamata for $f_*\omega^{\otimes m}_{X/Y}$ when $m\geq2$ \cite{Kaw82} and to the case that $\dim Y\geq2$ \cite{Kaw81}. Viehweg \cite{Vie82, Vie83} showed that $f_*\omega^{\otimes m}_{X/Y}$ is weakly positive by his mysterious covering trick. These positivity results are known even in the log-case \cite{Fujino18, KP17}. Similar results for higher direct images of $\omega_{X/Y}$ have also been established by Koll\'ar \cite{Kol1, Kol2}, Fujino \cite{Fujino04}, Fujino-Fujisawa \cite{FF14} and others. Zuo \cite{Zuo} showed the negativity of kernels of Kodaira-Spencer maps on Hodge bundles.

In this paper, we will introduce two new positivity notions for coherent sheaves that we call the $F^t$-ampleness and $F^tGG$-ampleness, here $t$ is a fixed positive integer. They are variants of the Frobenius amplitude introduced by Arapura \cite{Ara1} (see also \cite{Nakashima, Sun, Hart1}). Roughly speaking, the Frobenius amplitude cohomologically measure the positivity of a coherent sheaf $\mathcal{E}$ on $Y$ under the $t$-iterated Frobenius pull backs $(F_Y^t)^*$ for $t\gg0$. This
definition is most natural in positive characteristic, but our positivity notions are only defined in characteristic zero. They
measure the positivity of $(F^t_{Y_p})^*\mathcal{E}_p$ under a general reduction modulo $p$ of $(Y, \mathcal{E})$. Hence Arapura's positivity notions can be considered as the limits of ours. We also introduce Frobenius type semipositivity by relaxing the conditions for $F^t$-ampleness and $F^tGG$-ampleness.

We now list our positivity notions (see Definition \ref{def2.2}, \ref{def2.7}, \ref{defFt} and \ref{defFtGG} for the detailed definitions).
Let $\mathcal{E}$ be a coherent sheaf and $\mathcal{O}_Y(1)$ a very ample line bundle on $Y$.
\begin{enumerate}
  \item We say $\mathcal{E}$ is $F^t$-ample if for any locally free sheaf $\mathcal{F}$ we have $$H^i(Y_p, \mathcal{F}_p\otimes(F^t_{Y_p})^*\mathcal{E}_p)=0$$ for any $i>0$ and a general reduction modulo $p$ of $(Y, \mathcal{O}_Y(1), \mathcal{E}, \mathcal{F})$.
  \item We say $\mathcal{E}$ is $F^tGG$-ample if for any locally free sheaf $\mathcal{F}$, we have $$\mathcal{F}_p\otimes(F^t_{Y_p})^*\mathcal{E}_p$$ is globally generated for a general reduction modulo $p$.
  \item We say $\mathcal{E}$ is $F^t$-semipositive if there is an $m_0$ such that $$H^i(Y_p, \mathcal{O}_{Y_p}(m)\otimes(F^t_{Y_p})^*\mathcal{E}_p)=0$$ for any $i>0$, $m>m_0$ and a general reduction modulo $p$.
   \item We say $\mathcal{E}$ is (weakly) $F^tGG$-semipositive if there is an $m$ such that $$\mathcal{O}_{Y_p}(m)\otimes(F^t_{Y_p})^*\mathcal{E}_p$$ is (generically) globally generated for a general reduction modulo $p$.
\end{enumerate}
These new Frobenius type positivity notions can be consider as a generalization of the ordinary ampleness and nefness. We have the following implications (Lemma \ref{ample}, \ref{lemma3.8} and Proposition \ref{nef}):
\begin{enumerate}
  \item $\mathcal{E}$ is locally free and $F^tGG$-ample $\Rightarrow$ $\mathcal{E}$ is ample;
  \item $\mathcal{E}$ is locally free and $F^tGG$-semipositive $\Rightarrow$ $\mathcal{E}$ is nef;
  \item $\mathcal{E}$ is weakly $F^tGG$-semipositive $\Rightarrow$ $\mathcal{E}$ is weakly positive;
  \item $\mathcal{E}$ is $F^t$-ample $\Rightarrow$ $\mathcal{E}$ is $F^t$-semipositive $\Rightarrow$ $\mathcal{E}$ is $F^tGG$-semipositive;
  \item $\mathcal{E}$ is $F^t$-ample $\Rightarrow$ $\mathcal{E}$ is $F^tGG$-ample $\Rightarrow$ $\mathcal{E}$ is $F^tGG$-semipositive.
\end{enumerate}
Example \ref{eg2.13} and \ref{eg3.16} show that an ample vector bundle may not be weakly $F^tGG$-semipositive. Thus our positivity notions are stronger than the ordinary ones. They also share many properties with ampleness and nefness. They behave well under pullbacks (Theorem \ref{pullback} and Lemma \ref{pullback2}), restrictions (Theorem \ref{restriction}) and tensor products (Theorem \ref{tensor2}).

Exploiting the method of Deligne-Illusie, we obtain Kodaira-Akizuki-Nakano type vanishing theorem (Theorem \ref{thm2.13}) and Bott type vanishing theorem (Theorem \ref{Bott}) for $F^t$-ample sheaves. For $F^t$-semipositive sheaves we prove Fujita type vanishing theorems (Theorem \ref{Fujita1} and \ref{Fujita2}) and Kawamata-Viehweg type vanishing theorem (Theorem \ref{KV}). The numerical properties of $F^t$-semipositive sheaves are also studied. We obtain the positivity of high Chern characters and Bogomolov type inequality of $F^t$-semipositive sheaves (Theorem \ref{Chern1} and \ref{Bog}).

\subsection*{Main results}
The main results of this paper are the Frobenius type positivity of Hodge bundles. Now we state them.
\begin{theorem}\label{main}
Let $f: X\rightarrow Y$ be a semistable morphism between complex smooth projective varieties, and $\mathcal{L}$ a ample line bundle on $X$. Then
\begin{enumerate}
  \item $R^if_*\omega_{X/Y}$ is $F^1GG$-semipositive for any $i\geq0$;
  \item $f_*(\omega_{X/Y}\otimes\mathcal{L})$ is $F^1GG$-ample.
\end{enumerate}
In particular, $R^if_*\omega_{X/Y}$ is nef and $f_*(\omega_{X/Y}\otimes\mathcal{L})$ is ample or zero.
\end{theorem}
Since an ample vector bundle can not be weakly $F^1GG$-semipositive, the theorem strengthens and generalizes various previous semipositivity results of Fujita, Kawamata, Koll\'ar, Viehweg, Fujino–Fujisawa and others. By Viehweg's covering trick, one can obtain weaker positivity of Hodge bundles without the semistable assumption:
\begin{corollary}\label{main-cor}
Let $f: X\rightarrow Y$ be a surjective morphism between complex smooth projective varieties. Let $\mathcal{L}$ be a semiample line bundle on $X$. Then $$(f_*(\omega_{X/Y}\otimes\mathcal{L}))^{**}$$ is weakly $F^1GG$-semipositive. In particular, $f_*(\omega_{X/Y}\otimes\mathcal{L})$ is weakly $F^1GG$-semipositive if $f$ is flat.
\end{corollary}
This strengthens \cite[Proposition 2.43]{Vie95} of Viehweg. When $Y$ is a toric variety, we obtain stronger positivity of Hodge bundles.
\begin{theorem}\label{main-toric}
Let $f: X\rightarrow Y$ be a smooth morphism between complex smooth projective varieties. Assume further that $Y$ is a toric variety.
Then $R^jf_*\Omega^i_{X/Y}$ is $F^1$-semipositive for any $i,j\geq0$.
\end{theorem}

As a corollary, we obtain the following Bott type vanishing theorem and positivity of Chern characters.
\begin{corollary}\label{main-toric-cor}
Under the situation of Theorem \ref{main-toric}, we have $$H^b(Y, \Omega^a_Y\otimes R^jf_*\Omega^i_{X/Y}\otimes \mathcal{L})=0$$ for any $b>0$ and ample line bundle $\mathcal{L}$. And for any closed subvariety $D\subset Y$ of dimension $e$, we have $D\cdot\ch_e(R^jf_*\Omega^i_{X/Y})\geq0$.
\end{corollary}
One sees that Bott's vanishing follows when we take $i=j=0$ and $f$ to be the identity morphism of $Y$. When the cotangent bundle of $Y$ is $F^1$-semipositive, one can also get $F^1$-semipositivity of Hodge bundles. This assumption is satisfied if $Y$ is an abelian variety or a product variety of non-rational curves (see Example \ref{eg6.17}).
\begin{theorem}\label{main1}
Let $f: X\rightarrow Y$ be an $E$-semistable morphism between complex smooth projective varieties. If $\Omega^1_Y$ is $F^1$-semipositive, then both $R^if_*\omega_X$ and $R^if_*\omega_X(E_X)$ are $F^1$-semipositive.
\end{theorem}

As a corollary, we obtain Kawamata-Viehweg-Koll\'ar type vanishing theorems.
\begin{corollary}\label{main1-cor}
Under the situation of the theorem above, we have
\begin{eqnarray*}
&& H^b(X, \Omega_Y^a(\log E)(\lceil A\rceil)\otimes R^if_*\omega_X)\\
&=& H^b(X, \Omega_Y^a(\log E)(\lceil A\rceil-E)\otimes R^if_*\omega_X)\\
  &=&0
\end{eqnarray*}
for $a+b>\dim Y$, where $A$ is an ample $\mathbb{Q}$-divisor on $Y$ such that $\lceil A\rceil-A\leq E$. And for any closed subvariety $D\subset Y$ of dimension $e$, we have $$D\cdot\ch_e(R^if_*\omega_X)\geq0,~ D\cdot\ch_e(\mathcal{O}_Y(E)\otimes R^if_*\omega_X)\geq0.$$
\end{corollary}
In general $f_*\omega_{X/Y}$ may not be $F^1$-semipositive under the situation of Theorem \ref{main} (see Remark \ref{Mumford}). The main tools in the proof of Theorem \ref{main}, \ref{main-toric}, and \ref{main1} are Illusie's decomposition of the relative de Rham complexes $\Omega^{\bullet}_{X/Y}(\log E_X/E)$ and de Rham complexes with coefficients in the Gauss-Manin systems $\Omega^{\bullet}_{Y}(\log E)(\mathbb{H})$.

\subsection*{Applications}
Our positivity results can be used to investigate the image of Fano varieties under semistable morphisms and the slope inequality for a family of semistable curves.
\begin{theorem}[Theorem \ref{Fano}]\label{App-Fano}
Let $f: X\rightarrow Y$ be a surjective morphism between
complex smooth projective varieties. If $f$ is semistable and $-K_X$ is ample, then $-K_Y$ is also ample. If $-K_X$ is semiample, then $-K_Y$ is pseudoeffective.
\end{theorem}
Koll\'ar, Miyaoka and Mori \cite[Corollary 2.9]{KMM} proved that, under the assumption that $f$ is smooth, if $-K_X$ is ample, then so is $-K_Y$. Hence Theorem \ref{Fano} shows their result under a weaker assumption. We refer to \cite{BC, CZ, FG12, FG14} for the related results when $-K_X$ has weaker positivity.

Let $f: X\rightarrow Y$ be a family of curves of genus $g$ from a normal variety $X$ of dimension $n$ to a smooth variety $Y$.
Zhang \cite{Z1, Z2} proved the slope inequality
\begin{equation}\label{zhang}
K^n_{X/Y}\geq 2n!\left(\frac{g-1}{g+n-2}\right)\ch_{n-1}(f_*\omega_{X/Y})
\end{equation}
under one of the following assumptions:
\begin{enumerate}
  \item $K_{X/Y}$ is nef and $Y$ is a surface.
  \item $K_{X/Y}$ is nef, $f$ is semistable and $Y$ is the image of a finite morphism from an abelian
variety or a smooth toric Fano variety.
\end{enumerate}
Using the $F^1$-semipositivity of Hodge bundles, we can give the following inequality.
\begin{theorem}[Theorem \ref{slope}]\label{App-slope}
Let $f: X\rightarrow Y$ be a semistable morphism over $\mathbb{C}$ with $\dim X=n$. Assume the general fiber of $f$ is a smooth curve of genus $g$, $K_X$ is nef and $\Omega^1_Y$ is $F^1$-semipositive, then $$K^n_X\geq 2n!\left(\frac{g-1}{g+n-2}\right)\ch_{n-1}(f_*\omega_X).$$
\end{theorem}
In Remark \ref{Mumford}, one sees that $\ch_{n-1}(f_*\omega_{X/Y})$ may be negative or zero. Hence the inequality (\ref{zhang}) could be trivial in some cases. Our assumption that $\Omega^1_Y$ is $F^1$-semipositive guarantees the non-negativity of $\ch_{n-1}(f_*\omega_{X})$. Thus Theorem \ref{App-slope} seems nontrivial.

We pose the below question on the relation between Frobenius type positivity and Nakano positivity:
\begin{question}
Let $\mathcal{E}$ be a vector bundle on a complex smooth projective variety. Do the following implications hold?
\begin{enumerate}
  \item $\mathcal{E}$ is $F^t$-ample $\Rightarrow \mathcal{E}$ is Nakano positive;
  \item $\mathcal{E}$ is Nakano positive $\Rightarrow \mathcal{E}$ is $F^tGG$-ample.
\end{enumerate}
\end{question}

\subsection*{Organization of the paper}
Our paper is organized as follows. In Section \ref{S2}, we define the notions of $F^t$-ampleness and $F^tGG$-ampleness. We establish some of their basic properties. Then in Section \ref{S3} we introduce the notions of $F^t$-semipositivity and $F^tGG$-semipositivity and study their behaviour under pullbacks (Theorem \ref{pullback} and Lemma \ref{pullback2}), restrictions (Theorem \ref{restriction}) and tensor products (Theorem \ref{tensor2}). In Section \ref{S4} we prove Fujita type vanishing theorems (Theorem \ref{Fujita1} and \ref{Fujita2}) and Kawamata-Viehweg type vanishing theorem (Theorem \ref{KV}) for $F^t$-semipositive sheaves. Then we investigate the numerical properties of $F^t$-semipositive sheaves in Section \ref{S5}. We obtain the positivity of high Chern characters and Bogomolov type inequality of $F^t$-semipositive sheaves (Theorem \ref{Chern1} and \ref{Bog}). The main results and their applications will be proved in Section \ref{S6}.

\subsection*{Notation}
Given a field $k$, we write $\chr(k)$ for its characteristic and $W_2(k)$ the ring of Witt vectors of length two of $k$. Let $X$ be a smooth projective variety defined over $k$. The function field of $X$ is denoted by $K(X)$. We denote by $T_X$ and $\Omega_X^1$ the tangent bundle and cotangent bundle of $X$, respectively. $K_X$ and $\omega_X$ denote the canonical divisor and canonical sheaf of $X$, respectively. For a scheme $S$, we denote by $\D(S)$ the derived category of coherent sheaves of $\mathcal{O}_S$-modules on $S$. For a $\mathbb{Q}$-divisor $D$ on $X$, we write $\lceil D\rceil$ for the round-up of $D$.
We write $\ch(\mathcal{E})$ and $c(\mathcal{E})$ for the Chern character and Chern class
of a coherent sheaf $\mathcal{E}$ on $X$, respectively. We denote by $c_i(X):=c_i(T_X)$ for the $i$-th Chern class of $X$. We also write $H^j(\mathcal{F})$ ($j\in \mathbb{Z}_{\geq0}$) for the cohomology groups of a sheaf
$\mathcal{F}\in\Coh(X)$ and $h^j(\mathcal{F})$ for the dimension of $H^j(\mathcal{F})$. For a
sheaf $\mathcal{G}\in\Coh(X)$, we denote by $\mathcal{G}^*:=\mathcal{H}om(\mathcal{G}, \mathcal{O}_X)$ the
dual sheaf of $\mathcal{G}$ and by $\Delta(\mathcal{G}):=\ch^2_1(\mathcal{G})-2\ch_0(\mathcal{G})\ch_2(\mathcal{G})$
the discriminant of $\mathcal{G}$.

\subsection*{Acknowledgments}

The author would like to acknowledge with gratitude the profound influence of Professor Eckart Viehweg. The author's interest in positivity, moduli spaces and Hodge theory originated from a series of lectures delivered by Professor Viehweg at the Chinese-German Conference on Complex Geometry (2006) and the Algebraic Geometry Summer School at East China Normal University (2007). His enthusiasm for mathematics and his visionary perspective on algebraic geometry have been a lasting source of inspiration throughout the author's career.

The author would also like to thank Professor Kang Zuo for his interest and valuable discussions. Finally, the author expresses his appreciation to Professor Zhijie Chen, Professor Rong Du and Professor Sheng-Li Tan for their help and encouragement over the years. The author was supported by National Natural Science Foundation of China (Grant No. 12371047, 11771294) and Natural Science Foundation of Shanghai (Grant No. 23ZR1447000).

\section{$F^t$-ample and $F^tGG$-ample sheaves}\label{S2}
In this section, we define the notions of $F^t$-amplitude and $F^tGG$-ampleness of coherent sheaves in characteristic zero for a fixed positive integer $t$ and establish some of their basic properties. We let $k$ be a field with $\chr(k)=0$ and $X$ be a projective scheme over $k$ throughout this section. Most results and proofs in this section are parallel to that of \cite{Ara1}.

A diagram over $\Spec k$ is defined to be a collection of $k$-schemes $X_i$, morphism $X_i\rightarrow X_j$, coherent sheaves $\mathcal{E}_{i,l}$ on $X_i$ and morphisms between the pullbacks and pushforwards of these sheaves. Given a diagram $D$ over $\Spec k$, an arithmetic thickening (or simply just thickening) of it is a finitely generated $\mathbb{Z}$-subalgebra $A\subset k$ and a diagram $\widetilde{D}$ over $\Spec A$, so that $D$ is isomorphic to the fiber product $\widetilde{D}\times_{\Spec A}\Spec k$. We say a thickening $\widetilde{D}_0\rightarrow\Spec A_0$ refines the thickening $\widetilde{D}\rightarrow\Spec A$ if there is a homomorphism $A\rightarrow A_0$ and an isomorphism between $\widetilde{D}_0$ and $\widetilde{D}\times_{\Spec A}\Spec A_0$.

We have the following standard facts (see \cite[Lemma 1.1]{Ara1}).

\begin{lemma}\label{lemma2.1}
Any diagram over $\Spec k$ has a thickening. Any two thickenings over $\Spec k$ have a common refinement.
\end{lemma}

For a thickening $(\widetilde{X}, \widetilde{\mathcal{E}})\rightarrow\Spec A$ of $(X, \mathcal{E})$, let $k_s=A/s$, $p_s=\chr(k_s)$ and $k_{\bar{s}}$ be the algebraic closure of $k_s$ for each closed point $s\in \Spec A$. We write $X_s$ for the fiber and $\mathcal{E}_s=\widetilde{\mathcal{E}}|_{X_s}$ for each closed point $s\in \Spec A$ or a geometric point $s$ lying over a closed point. We denote by $F_s$ the absolute Frobenius morphism of $X_s$ and denote $(F^t_s)^*\mathcal{E}_s$ by $\mathcal{E}_s^{(p^t_s)}$. We say that a property holds for almost all $s\in \Spec A$ if it holds for all closed point $s$ in a nonempty open subset of $\Spec A$.

We now give the notion of $F^t$-amplitude which provides a cohomological measure of the positivity of a sheaf.
\begin{definition}\label{def2.2}
Let $\mathcal{F}$ be a coherent sheaf on $X$. The $F^t$-amplitude $\phi^t_a(\mathcal{F})$ of $\mathcal{F}$ is the smallest nonnegative integer such that for any locally free sheaf $\mathcal{E}$ and thickening $(\widetilde{X}, \widetilde{\mathcal{F}}, \widetilde{\mathcal{E}})\rightarrow\Spec A$ of $(X, \mathcal{F}, \mathcal{E})$, there exists a nonempty open subset $U$ of $\Spec A$ such that
$H^i(X_s, \mathcal{E}_s\otimes\mathcal{F}_s^{(p^t_s)})=0$ for all $i>\phi^t_a(\mathcal{F})$ and closed point $s\in U$. We say $\mathcal{F}$ is $F^t$-ample if $\phi^t_a(\mathcal{F})=0$.
\end{definition}

Taking $t\rightarrow+\infty$, one obtains Arapura's definition of Frobenius amplitude in characteristic zero:

\begin{definition}\label{def2.3}
Let $\mathcal{F}$ be a coherent sheaf on $X$. The $F^{\infty}$-amplitude $\phi^{\infty}_a(\mathcal{F})$ of $\mathcal{F}$ is the smallest nonnegative integer such that for any locally free sheaf $\mathcal{E}$ and thickening of $(X, \mathcal{F}, \mathcal{E})$ over $\Spec A$, there exist a nonempty open subset $U$ of $\Spec A$ and integers $N_s$ for every $s\in U$ such that
$H^i(X_s, \mathcal{E}_s\otimes\mathcal{F}_s^{(p^n_s)})=0$ for all $i>\phi^{\infty}_a(\mathcal{F})$, closed point $s\in U$ and $n\geq N_s$. We say $\mathcal{F}$ is $F^{\infty}$-ample if $\phi^{\infty}_a(\mathcal{F})=0$.
\end{definition}

\begin{remark}\label{rm2.4}
In \cite{Ara1}, Arapura used the term $F$-amplitude for what we now call $F^{\infty}$-amplitude. The
reason for this change is to distinguish our notion $F^t$-amplitude from his more clearly. Our definition of $F^{\infty}$-amplitude appears to be somewhat different from Arapura's, but they are equivalent by Lemma \ref{lemma2.1}. An $F^{\infty}$-ample sheaf is also termed a cohomologically $p$-ample sheaf.
\end{remark}

\begin{remark}\label{rm2.5}
In the definitions of $F^{t}$-amplitude and $F^{\infty}$-amplitude we let $\bar{s}$ be the geometric point lying over $s$. By the proof of \cite[Lemma C.5]{Ara1} and \cite[Lemma 3.8]{Ke}, one sees that $$H^i(X_{\bar{s}}, \mathcal{E}_{\bar{s}}\otimes\mathcal{F}_{\bar{s}}^{(p^n_{\bar{s}})})=H^i(X_s, \mathcal{E}_s\otimes\mathcal{F}_s^{(p^n_s)})\otimes_{k_s}k_{\bar{s}}.$$
Therefore the cohomological conditions in these definitions can be tested for general geometric points.
\end{remark}

\begin{definition}
Let $U$ be an open subset of a scheme $Y$. A coherent sheaf $\mathcal{E}$ on $Y$ is said to be globally generated on $U$ if the evaluation map
$H^0(Y, \mathcal{E})\otimes\mathcal{O}_Y\rightarrow\mathcal{E}$ is surjective on $U$. We say $\mathcal{E}$ is globally generated if it is globally generated on $Y$, and $\mathcal{E}$ is generically globally generated if it is globally generated on a dense open subset of $Y$.
\end{definition}

One can also extend the concept of $p$-ampleness introduced by Hartshorne \cite{Hart1} to characteristic zero using thickenings:
\begin{definition}\label{def2.7}
Let $\mathcal{F}$ be a coherent sheaf on $X$. We say $\mathcal{F}$ is $F^tGG$-ample if for any locally free sheaf $\mathcal{E}$ and thickening $(\widetilde{X}, \widetilde{\mathcal{F}}, \widetilde{\mathcal{E}})\rightarrow\Spec A$ of $(X, \mathcal{F}, \mathcal{E})$, there exists a nonempty open subset $U$ of $\Spec A$ such that $\mathcal{E}_s\otimes\mathcal{F}_s^{(p^t_s)}$ is globally generated for every closed point $s\in U$.
\end{definition}

From the definition, one sees immediately that any quotient sheaf of an $F^tGG$-ample sheaf is $F^tGG$-ample. The following lemma shows that $F^t$-amplitude and $F^tGG$-ampleness can be only tested by tensoring with ample line bundles.
\begin{lemma}\label{lemma2.6}
Let $\mathcal{L}$ be an ample line bundle on $X$ and $\mathcal{F}$ be a coherent sheaf on $X$.
Then $\phi^t_a(\mathcal{F})\leq N$ if and only if for any integer $m$ and thickening of $(X, \mathcal{F}, \mathcal{L})$ over $\Spec A$, there exists a nonempty open subset $U$ of $\Spec A$ such that $$H^i(X_s, \mathcal{L}_s^m\otimes\mathcal{F}_s^{(p^t_s)})=0$$ for all $i>N$ and closed point $s\in U$. The analogical statement holds for $F^tGG$-ampleness of $\mathcal{F}$.
\end{lemma}
\begin{proof}
The proof is similar to that of \cite[Corollary 2.3]{Ara1}. The only-if-direction is obvious. To prove the if-direction, we take a locally free sheaf $\mathcal{E}$ on $X$ and let $d=\dim X$.
By Serre's theorem, we can obtain exact sequences of locally free sheaves:
\begin{eqnarray*}
&0\rightarrow\mathcal{E}_1\rightarrow V_0\otimes\mathcal{L}^{-m_0}\rightarrow\mathcal{E}\rightarrow0,&\\
&0\rightarrow\mathcal{E}_2\rightarrow V_1\otimes\mathcal{L}^{-m_1}\rightarrow\mathcal{E}_1\rightarrow0,&\\
&\vdots&\\
&0\rightarrow\mathcal{E}_{d+1}\rightarrow V_d\otimes\mathcal{L}^{-m_d}\rightarrow\mathcal{E}_d\rightarrow0,&
\end{eqnarray*}
where $V_i$ is a vector space and $m_i$ is a positive integer for $0\leq i\leq d$.

Choose a thickening of these sequences over $\Spec A$. For a closed point $s\in\Spec A$, we consider the restrictions of these sequences on the fiber $X_s$ and tensor them with $\mathcal{F}_s^{(p^t_s)}$. Applying the long exact sequence for cohomology shows that $$H^i(X_s, \mathcal{E}_s\otimes\mathcal{F}_s^{(p^t_s)})=0$$ for all $i>N$ and for almost all $s\in \Spec A$. Hence $\phi^t_a(\mathcal{F})\leq N$.
\end{proof}

The proof also gives the following:
\begin{lemma}\label{lemma2.7}
Let $\mathcal{F}$ be a locally free sheaf and $\mathcal{E}$ a coherent sheaf on $X$. Then for any thickening of $(X, \mathcal{F}, \mathcal{E})$ over $\Spec A$, we have $H^i(X_s, \mathcal{E}_s\otimes\mathcal{F}_s^{(p^t_s)})=0$ for $i>\phi^t_a(\mathcal{F})$ and almost all $s\in \Spec A$. The analogical statement holds for $F^tGG$-ample locally free sheaves.
\end{lemma}

The concept of $F^t$-amplitude and $F^tGG$-ampleness behave well under base change.
\begin{lemma}\label{lemma2.8}
Let $K$ be an extension of $k$ and $\mathcal{F}$ a coherent sheaf on $X$. Then $\mathcal{F}$ is $F^tGG$-ample if and only if $\mathcal{F}\otimes_kK$ is $F^tGG$-ample, and $\phi^t_a(\mathcal{F})=\phi^t_a(\mathcal{F}\otimes_kK)$.
\end{lemma}
\begin{proof}
Let $\mathcal{L}$ be an ample line bundle on $X$. Let $$f: X^{\prime}:=X\times_{\Spec k}\Spec K\rightarrow X$$ be the base change. Then $f^*\mathcal{L}$ is also ample on $X^{\prime}$. Choose a thickening $(\widetilde{X}, \widetilde{\mathcal{F}}, \widetilde{\mathcal{L}})$ of $(X, \mathcal{F}, \mathcal{L})$ over $\Spec A$. Since $A\subset k\subset K$. one sees that $(\widetilde{X}, \widetilde{\mathcal{F}}, \widetilde{\mathcal{L}})$ is also a thickening of $(X^{\prime}, f^*\mathcal{F}, f^*\mathcal{L})$. Hence the conclusion follows from the definition of $F^tGG$-ampleness and $\phi^t_a$.
\end{proof}

The $F^t$-ampleness and ampleness are equivalent for line bundles:
\begin{lemma}\label{lemma2.9}
A line bundle $\mathcal{L}$ on $X$ is $F^t$-ample if and only if it is ample.
\end{lemma}
\begin{proof}
Let $\mathcal{E}$ be a locally free sheaf on $X$. If $\mathcal{L}$ is ample, by Serre's vanishing theorem one sees that there exists an integer $N$ such that
$$H^i(X, \mathcal{L}^{n}\otimes\mathcal{E})=0$$ for any $i>0$ and $n>N$. Taking a thickening of $(X, \mathcal{L}, \mathcal{E})$ over $\Spec A$, we have
$$H^i(X_s, \mathcal{L}_s^{(p_s^t)}\otimes\mathcal{E}_s)=0$$ for $i>0$ and for a general closed point $s\in\Spec A$ with $p_s^t>N$. Therefore $\mathcal{L}$ is $F^t$-ample.

Now we assume that $\mathcal{L}$ is $F^t$-ample. Let $\mathcal{O}_X(1)$ be a very ample line bundle on $X$. Choosing a thickening of $(X, \mathcal{L}, \mathcal{O}_X(1))$ over $\Spec A$, one sees that $$H^i(X_s, \mathcal{L}_s^{(p_s^t)}(-i-1))=0$$ for $i>0$ and a general $s\in\Spec A$. Hence $\mathcal{L}_s^{(p_s^t)}(-1)$ is globally generated by \cite[Theorem 1.8.5]{Laz}. This implies that $\mathcal{L}_s^{(p_s^t)}$ is ample. Therefore $\mathcal{L}_s$ is ample for general $s\in\Spec A$. So is $\mathcal{L}$.
\end{proof}

The proof also gives:
\begin{lemma}\label{ample}
Let $\mathcal{E}$ be a coherent sheaf on $X$.
\begin{enumerate}
  \item If $\mathcal{E}$ is $F^t$-ample, then it is $F^tGG$-ample.
  \item If $\mathcal{E}$ is $F^tGG$-ample and locally free, then it is ample.
\end{enumerate}
\end{lemma}

The following is a characterization of $F^t$-ampleness for curves.
\begin{lemma}\label{lemma2.11}
Assume that $X$ is a smooth projective curve defined over $k$. Let $\mathcal{F}$ be a locally free sheaf on $X$, then the following are equivalent
\begin{enumerate}
  \item $\mathcal{F}$ is $F^t$-ample;
  \item $\mathcal{F}$ is $F^tGG$-ample;
  \item $\mathcal{F}$ is ample;
  \item $\mathcal{F}$ is $F^{\infty}$-ample.
\end{enumerate}
\end{lemma}
\begin{proof}
The equivalence of the last two statements follows from \cite[Prop. 5.4]{Ara1}. By Lemma \ref{ample}, we only need to show $(3) \Rightarrow (1)$. We may assume that $k$ is algebraically closed from Lemma \ref{lemma2.8}.

Suppose that $\mathcal{F}$ is ample. We shall show $\mathcal{F}$ is $F^t$-ample. Let $\mathcal{L}$ be an ample line bundle on $X$ and $(\widetilde{X}, \widetilde{\mathcal{F}}, \widetilde{\mathcal{L}})$ be a thickening of $(X, \mathcal{F}, \mathcal{L})$ over $\Spec A$. For a general closed point $s$ of $\Spec A$, by Langer's estimation (see Theorem \ref{Langer2}), one sees that
$$\mu^{-}(\mathcal{F}_{s})-\frac{\mu^{-}(\mathcal{F}_{s}^{(p_{s}^t)})}{p_{s}^t}\leq\frac{r-1}{p_{s}-1}N,$$ here $r=\rank \mathcal{F}$ and $N$ is a constant depending only on the genus of $X$. Thus one infers $$\mu^{-}(\mathcal{F}_{s}^{(p_{s}^t)})\geq p_{s}^t\left(\mu^{-}(\mathcal{F}_{s})-\frac{r-1}{p_{s}-1}N\right)
=p_{s}^t\left(\mu^{-}(\mathcal{F})-\frac{r-1}{p_{s}-1}N\right).$$ Since $\mu^{-}(\mathcal{F})>0$, one sees that $\mu^{-}(\mathcal{F}_{s}^{(p_{s}^t)})$ tends to infinity when $p_{s}\rightarrow\infty$.
It follows that
\begin{eqnarray*}
H^1(X_{s}, \mathcal{L}_{s}^m\otimes\mathcal{F}_{s}^{(p^t_{s})})
=\Hom(\mathcal{F}_{s}^{(p^t_{s})}, \mathcal{L}_{s}^{-m}\otimes\omega_{X_{s}})^{*}=0
\end{eqnarray*}
for almost all $s$. Therefore $\mathcal{F}$ is $F^t$-ample by Lemma \ref{lemma2.6}.
\end{proof}

In general an $F^tGG$-ample coherent sheaf may not be $F^t$-ample, and an ample locally free sheaf may not be $F^tGG$-ample.
\begin{example}\label{eg2.13}
Let $Y$ be a smooth complex projective variety of of dimension $n$, let $H$ be a very ample divisor on $Y$, and let $\mathcal{F}_0$ be any vector bundle on $Y$ of rank $n+f$ with $f\geq n$. By \cite[Theorem 6.3.65]{Laz} one sees that for any $d\gg0$ the cokernel of a sufficiently general vector bundle map $$u:\mathcal{O}_Y(-dH)^{\oplus n}\rightarrow \mathcal{F}_0$$ is an ample vector bundle of rank $f$.

Taking $Y=\mathbb{P}^2$, for $d\gg0$ there are ample vector bundles $\mathcal{E}_1$ and $\mathcal{E}_2$ having presentations of the following forms:
\begin{eqnarray*}
 &&0\rightarrow\mathcal{O}_{\mathbb{P}^2}(-d)^{\oplus2}\rightarrow\mathcal{O}_{\mathbb{P}^2}(-1)^{\oplus4}\rightarrow\mathcal{E}_1\rightarrow0;\\
 &&0\rightarrow\mathcal{O}_{\mathbb{P}^2}(-d)^{\oplus2}\rightarrow\mathcal{O}_{\mathbb{P}^2}(1)^{\oplus4}\rightarrow\mathcal{E}_2\rightarrow0.
\end{eqnarray*}
One notices that $\mathcal{O}_{\mathbb{P}^2}(1)^{\oplus4}$ is $F^tGG$-ample, so is $\mathcal{E}_2$. Choose a thickening of these two sequences over $\Spec A$. For a general closed point $s\in \Spec A$, we consider the $t$-iterated Frobenius pull backs of the restrictions of these exact sequences on the fiber $Y_s$. The long exact sequences for cohomology give $h^0(Y_s, \mathcal{E}^{(p_s^t)}_{1,s})=0$ and
$$h^1(Y_s, \mathcal{E}^{(p_s^t)}_{2,s})=2h^2(\mathbb{P}^2, \mathcal{O}_{\mathbb{P}^2}(-dp_s^t))=2h^0(\mathbb{P}^2, \mathcal{O}_{\mathbb{P}^2}(dp_s^t-3))>0$$ for almost all $s$. Therefore $\mathcal{E}_1$ is not $F^tGG$-ample, and $\mathcal{E}_2$ is not $F^t$-ample.
\end{example}

\begin{theorem}\label{thm2.12}
Let $\mathcal{E}, \mathcal{E}_i, \mathcal{E}^j$ be coherent sheaves on $X$, where $1\leq i, j\leq m$. Assume either that $X$ is smooth or that these sheaves are locally free. Then the following statements hold.
\begin{enumerate}
  \item Given an exact sequence $0\rightarrow\mathcal{E}_1\rightarrow\mathcal{E}\rightarrow\mathcal{E}_2\rightarrow0$, we have $\phi^t_{a}(\mathcal{E})\leq\max\{\phi^t_{a}(\mathcal{E}_1), \phi^t_{a}(\mathcal{E}_2)\}$.
  \item Let $0\rightarrow\mathcal{E}_m\rightarrow\cdots\rightarrow\mathcal{E}_0\rightarrow\mathcal{E}\rightarrow0$ be an exact sequence such that $\phi^t_{a}(\mathcal{E}_i)\leq i+l$ for each $i$, then $\phi^t_{a}(\mathcal{E})\leq l$.
  \item Let $0\rightarrow\mathcal{E}\rightarrow\mathcal{E}^0\rightarrow\cdots\rightarrow\mathcal{E}^m\rightarrow0$ be an exact sequence such that $\phi^t_{a}(\mathcal{E}^i)\leq l-i$ for each $i$, then $\phi^t_{a}(\mathcal{E})\leq l$.
  \item Let $f: Y\rightarrow X$ be a proper morphism of projective schemes such that $d$ is the maximum dimension of the closed fibers. If $\mathcal{E}$ is locally free then $\phi^t_{a}(f^*\mathcal{E})\leq \phi^t_{a}(\mathcal{E})+d$.
  \item If $f: Y\rightarrow X$ is an \'etale morphism of smooth projective schemes, then $\phi^t_{a}(f_*\mathcal{E})= \phi^t_{a}(\mathcal{E})$.
\end{enumerate}
\end{theorem}
\begin{proof}
The proof is similar to that of \cite[Theorem 1]{Ara1}. We leave the details to the interested reader.
\end{proof}

We give the following estimate on the $F^t$-amplitude of ample vector bundles, which is a variant of \cite[Theorem 5]{Ara1}.

\begin{theorem}\label{est-ample}
Assume that $X$ is integral. Let $\mathcal{E}$ be an ample vector bundle on $X$ of rank $r$. Then $\phi^t_a(\mathcal{E})<r$.
\end{theorem}
\begin{proof}
The proof is the same as that of \cite[Theorem 5]{Ara1}.
\end{proof}

Exploiting the method of Deligne-Illusie, one obtains Kodaira-Akizuki-Nakano type vanishing theorem. It can also be considered as a variant of \cite[Cor. 8.5]{Ara1}.
\begin{theorem}\label{thm2.13}
Assume that $X$ is smooth with $\dim X=d$. Let $D$ be a simple normal crossings divisor and $\mathcal{E}$ be a coherent sheaf on $X$. Then we have $$H^i(X, \Omega^j_X(\log D)(-D)\otimes \mathcal{E})=0$$ for
$i+j>d+\phi_a^{t}(\mathcal{E})$. In particular, one has $$H^i(X, \Omega^j_X\otimes \mathcal{E})=0$$ for
$i+j>d+\phi_a^{t}(\mathcal{E})$.
\end{theorem}
\begin{proof}
Choose a thickening $(\widetilde{X}, \widetilde{D}, \widetilde{\mathcal{E}})\rightarrow\Spec A$ of $(X, D, \mathcal{E})$. For a general closed point $s$ of $\Spec A$, there is a quasi-isomorphism $$\Omega^{\bullet}_{X_{s}}(\log D_s)\cong\Omega^{\bullet}_{X_{s}}(\log D_s)((p_s-1)D_s)$$ by \cite[3.3]{Hara}.
Tensoring both sides with $\mathcal{E}^{(p_{s}^{t})}_{s}(-p_sD_s)$ yields
$$\Omega^{\bullet}_{X_{s}}(\log D_s)\otimes F_s^*(\mathcal{E}^{(p_{s}^{t-1})}_{s}(-D_s))\cong\Omega^{\bullet}_{X_{s}}(\log D_s)\otimes\mathcal{E}^{(p_{s}^{t})}_{s}(-D_s).$$
From \cite[Lemma 5.6]{Ara1} and \cite[\S4.2]{DI}, it follows that
\begin{eqnarray*}
 F_{s,*}(\Omega^{\bullet}_{X_{s}}(\log D_s)\otimes\mathcal{E}^{(p_{s}^{t})}_{s}(-D_s)) &\cong& F_{s,*}(\Omega^{\bullet}_{X_{s}}(\log D_s))\otimes \mathcal{E}^{(p_{s}^{t-1})}_{s}(-D_s) \\
   &\cong&\Big(\bigoplus_{j=0}^n\Omega^j_{X_{s}}(\log D_s)[-j]\Big)\otimes\mathcal{E}^{(p_{s}^{t-1})}_{s}(-D_s).
\end{eqnarray*}
These isomorphisms together with the spectral sequence
$$H^i(\Omega^j_{X_{s}}(\log D_s)(-D_s)\otimes\mathcal{E}^{(p_{s}^t)}_{s})\Rightarrow
\mathbb{H}^{i+j}(\Omega^{\bullet}_{X_{s}}(\log D_s)(-D_s)\otimes\mathcal{E}^{(p_{s}^t)}_{s})$$
give that
\begin{eqnarray*}
 \bigoplus_{j=0}^d H^{i-j}(\Omega^j_{X_{s}}(\log D_s)(-D_s)\otimes\mathcal{E}^{(p_{s}^{t-1})}_{s})  &=& \mathbb{H}^i(F_{s,*}(\Omega^{\bullet}_{X_{s}}(\log D_s)(-D_s)\otimes\mathcal{E}^{(p_{s}^{t})}_{s})) \\
   &=& \mathbb{H}^i(\Omega^{\bullet}_{X_{s}}(\log D_s)(-D_s)\otimes\mathcal{E}^{(p_{s}^{t})}_{s})\\
   &\leq&\bigoplus_{j=0}^d H^{i-j}(\Omega^j_{X_{s}}(\log D_s)(-D_s)\otimes\mathcal{E}^{(p_{s}^{t})}_{s})\\
   &=&0
\end{eqnarray*}
for $i>d+\phi_a^{t}(\mathcal{E})$ and almost all $s$. Induction yields
$$\bigoplus_{j=0}^d H^{i-j}(\Omega^j_{X_{s}}(\log D_s)(-D_s)\otimes\mathcal{E}^{(p_{s}^{m})}_{s})=0$$ for any $i>d+\phi_a^{t}(\mathcal{E})$, $0\leq m\leq t$ and almost all $s$. The conclusion follows from the semicontinuity.
\end{proof}

We can also give a generalization of Bott's vanishing for toric schemes. We refer to \cite[3.3]{BTLM} for the definition of toric schemes.
\begin{theorem}\label{Bott}
Assume that $X$ is a smooth projective toric $k$-scheme. We have $$H^i(X, \Omega^j_X\otimes\mathcal{E})=0$$ for any $i>\phi_a^t(\mathcal{E})$ and $j\geq0$.
\end{theorem}
\begin{proof}
Choose a thickening $(\widetilde{X}, \widetilde{\mathcal{E}})$ over $\Spec A$. By \cite[Theorem 2]{BTLM}, we have a split sequence $$0\rightarrow\Omega^j_{X_s}\rightarrow F_{s,*}\Omega^j_{X_s}$$ for almost all $s$. Tensoring with $\mathcal{E}_s$, we obtain an injection
$$H^i(X_s, \Omega^j_{X_s}\otimes \mathcal{E}_s)\hookrightarrow H^i(X_s, \Omega^j_{X_s}\otimes\mathcal{E}^{(p_s)}_s)\hookrightarrow\cdots\hookrightarrow H^i(X_s, \Omega^j_{X_s}\otimes\mathcal{E}^{(p_s^t)}_s)=0$$ for almost all $s$. Hence one gets the desired vanishing by semicontinuity.
\end{proof}

\section{$F^t$-semipositive and $F^tGG$-semipositive sheaves}\label{S3}
Let $t$ be a fixed positive integer. We introduce, in this section, the notions of $F^t$-semipositivity and $F^tGG$-semipositivity and establish some basic properties. The notions are obtained by relaxing the conditions for $F^{t}$-ampleness and $F^tGG$-ampleness. We still let $k$ be a field with $\chr(k)=0$ and $X$ a projective scheme over $k$ in this section. We freely use the notations in Section \ref{S2}. Most results in this section are parallel to that of \cite{Sun}.

\subsection{Partial Castelnuovo–Mumford regularity}\label{S3.1}
Castelnuovo–Mumford regularity gives a quantitative measure of the algebraic complexity of a coherent sheaf. We will recall the definition and basic properties (e. g. \cite[sec. 2]{Ke}).

\begin{definition}
Let $\mathcal{O}_X(1)$ be a very ample line bundle on $X$, and $u$ a nonnegative integer. We say a coherent sheaf $\mathcal{E}$ on $X$ is $(m, u)$-regular if $$H^i(X, \mathcal{E}(m-i))=0$$ for $i>u$. The $u$-regularity $\reg_u(\mathcal{E})$ of $\mathcal{E}$ is the minimum $m$ such that $\mathcal{E}$ is $(m, u)$-regular. If $\mathcal{E}$ is $(m, 0)$-regular, we simply say $\mathcal{E}$ is $m$-regular.
\end{definition}

The $0$-regularity is the usual Castelnuovo–Mumford regularity (see \cite[sec. 1.8]{Laz}). when $u>0$, we say $u$-regularity partial Castelnuovo–Mumford regularity. Here we use $u$-regularity instead of just $0$-regularity because it is necessary for the study of $F^t$-amplitude. This concept of $u$-regularity shares a basic property with $0$-regularity.

\begin{lemma}\label{pCM}
Let $\mathcal{E}$ be a coherent sheaf on $X$. If $\mathcal{E}$ is $(m, u)$-regular, then $\mathcal{E}$ is $(n, u)$-regular for all
$n\geq m$.
\end{lemma}
\begin{proof}
The proof is essentially the same as that of $0$-regularity (cf. \cite[Lemma 2.2]{Ke}).
\end{proof}

The related notion of the level $\lambda(\mathcal{E})$ of a sheaf was introduced in \cite[sec. 1]{Ara2}. It measures the deviation from $0$-regular sheaves. The following lemma is helpful for us to compute the cohomology groups of a sheaf.

\begin{lemma}\label{lemma3.3}
Let $\mathcal{E}$ be a $m$-regular sheaf on $X$ with respect to a very ample line bundle $\mathcal{O}_X(1)$. Then for any $N\geq0$, there exist vector spaces $V_i$ and a resolution
$$V_N\otimes\mathcal{O}_X(-m-NR)\rightarrow\cdots \rightarrow V_1\otimes\mathcal{O}_X(-m-R)\rightarrow V_0\otimes\mathcal{O}_X(-m)\rightarrow\mathcal{E}\rightarrow0$$ where $R=\max\{1, \reg_0(\mathcal{O}_X)\}$. In particular, a $0$-regular sheaf is globally generated.
\end{lemma}
\begin{proof}
See \cite[Corollary 3.2]{Ara1}.
\end{proof}

\begin{proposition}\label{tensor-CM}
Let $\mathcal{O}_X(1)$ be a very ample line bundle on $X$. Let $\mathcal{E}$ and $\mathcal{F}$ be two coherent sheaves on $X$ such that one of them is locally free, then
$$\reg_u(\mathcal{E}\otimes\mathcal{F})\leq\reg_0(\mathcal{E})+\reg_u(\mathcal{F})+(\dim X-u-1)(R-1)$$ where $R=\max\{1, \reg_0(\mathcal{O}_X)\}$.
\end{proposition}
\begin{proof}
See \cite[Prop. 2.8]{Ke}.
\end{proof}

\subsection{Basic definitions and properties}

\begin{definition}\label{defFt}
Let $\mathcal{O}_X(1)$ be a very ample line bundle on $X$ and $\mathcal{E}$ a coherent sheaf on $X$. The $F^t$-semipositivity $\phi^t_{sp}(\mathcal{E})$ of $\mathcal{E}$ is the smallest nonnegative integer such that for any thickening $(\widetilde{X}, \mathcal{O}_{\widetilde{X}}(1), \widetilde{\mathcal{E}})\rightarrow\Spec A$ of $(X, \mathcal{O}_X(1), \mathcal{E})$, there exists a nonempty open subset $U$ of $\Spec A$ such that the $\phi^t_{sp}(\mathcal{E})$-regularity of $\mathcal{E}_s^{(p_s^t)}$ with respect to $\mathcal{O}_{X_s}(1)$ are bounded for all closed point $s\in U$.  We say $\mathcal{E}$ is $F^t$-semipositive if $\phi^t_{sp}(\mathcal{E})=0$.
\end{definition}

\begin{remark}\label{rm3.6}
It follows from Lemma \ref{pCM} and Lemma \ref{lemma2.1} that $\phi^t_{sp}(\mathcal{E})\leq N$ if and only if there exists an integer $m$ so that $H^{j}(\mathcal{E}_s^{(p_s^t)}(i))=0$ for any $i\geq m$, $j>N$ and almost all $s$ for a given thickening.
\end{remark}

Taking $t\rightarrow+\infty$, one obtains the definition of Frobenius semipositivity in \cite{Sun}:

\begin{definition}
Let $\mathcal{O}_X(1)$ be a very ample line bundle on $X$ and $\mathcal{E}$ a coherent sheaf on $X$. The $F^{\infty}$-semipositivity $\phi^{\infty}_{sp}(\mathcal{E})$ of $\mathcal{E}$ is the smallest nonnegative integer such that for any thickening of $(X, \mathcal{O}_X(1), \mathcal{E})$ over $\Spec A$, there exists a nonempty open subset $U$ of $\Spec A$ and integers $N_s$ for every $s\in U$ such that the $\phi^{\infty}_{sp}(\mathcal{E})$-regularity of $\mathcal{E}_s^{(p_s^n)}$ with respect to $\mathcal{O}_{X_s}(1)$ are bounded for all closed point $s\in U$ and $n\geq N_s$. We say $\mathcal{F}$ is $F^{\infty}$-semipositive if $\phi^{\infty}_{sp}(\mathcal{F})=0$.
\end{definition}

We used in \cite{Sun} the term $F$-semipositivity for what we now call $F^{\infty}$-semipositivity. The reason for this change is also to distinguish the notion $F^t$-semipositivity from $F$-semipositivity more clearly. Relaxing the condition for $F^tGG$-ampleness, one obtains the notion of $F^tGG$-semipositive sheaves and weakly $F^tGG$-semipositive sheaves:

\begin{definition}\label{defFtGG}
Let $U$ be an open subscheme of a quasi-projective $k$-scheme $Y$. Let $\mathcal{O}_Y(1)$ be a very ample line bundle on $Y$ and $\mathcal{E}$ a coherent sheaf on $Y$. We say $\mathcal{E}$ is weakly $F^tGG$-semipositive on $U$, if for any thickening $(\widetilde{U}\subset\widetilde{Y}, \mathcal{O}_{\widetilde{Y}}(1), \widetilde{\mathcal{E}})$ over $\Spec A$, there exists an integer $m$ such that $\mathcal{E}_s^{(p_s^t)}\otimes\mathcal{O}_{X_s}(m)$ is globally generated on $U_s$ for almost all $s\in \Spec A$. We say $\mathcal{E}$ is weakly $F^tGG$-semipositive if $\mathcal{E}$ is weakly $F^tGG$-semipositive on some dense open subscheme. $\mathcal{E}$ is called $F^tGG$-semipositive if $\mathcal{E}$ is weakly $F^tGG$-semipositive on $Y$.
\end{definition}


The proposition below is an analog of \cite[Cor. 3.9]{Ara1} and \cite[Prop. 4.4]{Sun}.
\begin{proposition}
 The above definitions are independent of the choice of very ample line bundles.
\end{proposition}
\begin{proof}
Let $\mathcal{E}$ be a coherent sheaf on $X$. Let $\mathcal{L}_1$ and $\mathcal{L}_2$ be two very ample line bundles on $X$. We write $\phi^{t}_{sp, i}(\mathcal{E})$ for the $F^t$-semipositivity of $\mathcal{E}$ with respect to $\mathcal{L}_i$, where $i=1$ or $2$. Let $R_0$ be the $0$-regularity of  $\mathcal{O}_X$ with respect to $\mathcal{L}_2$ and $d=\dim X$.

Choose a thickening of $(X, \mathcal{L}_1, \mathcal{L}_2, \mathcal{E})$ over $\Spec A$. By the definition of $\phi^{t}_{sp, 2}(\mathcal{E})$, there exists an integer $m_2$ such that $$H^i(X_s, \mathcal{E}_s^{(p_s^t)}\otimes\mathcal{L}^{m-i}_{2, s})=0$$ for any $m\geq m_2$, $i>\phi^{t}_{sp, 2}(\mathcal{E})$ and almost all $s$. By Serre's vanishing theorem,
we can take an integer $m_1$ such that the $0$-regularity of $\mathcal{L}^{m}_1$ with respect to $\mathcal{L}_2$ is less than $-m_2-dR$ for any $m\geq m_1-d$, where $R=\max\{1, R_0\}$. By semicontinuity and Lemma \ref{lemma3.3}, for any $1\leq j\leq d$ and general $s\in\Spec A$ we obtain an exact sequence $$0\rightarrow\mathcal{E}_{d+1,s}\rightarrow\mathcal{E}_{d, s}\rightarrow\cdots\rightarrow\mathcal{E}_{0,s}\rightarrow\mathcal{L}^{m_1-j}_{1, s}\rightarrow0,$$
where $\mathcal{E}_{i, s}=V_{i}\otimes\mathcal{L}^{m_2+dR-iR}_{2, s}$ for $i\leq d$. Tensoring this by $\mathcal{E}_s^{(p_s^t)}$ yields an exact complex
$$\cdots\rightarrow\mathcal{E}_{1, s}\otimes\mathcal{E}_s^{(p_s^t)}\rightarrow\mathcal{E}_{0, s}\otimes\mathcal{E}_s^{(p_s^t)}
\rightarrow\mathcal{L}_{1, s}^{m_1-j}\otimes\mathcal{E}_s^{(p_s^t)}\rightarrow0$$ for almost all $s$.
By chasing through this exact complex, one obtains $$H^j(X_s, \mathcal{E}_s^{(p_s^t)}\otimes\mathcal{L}^{m_1-j}_{1, s})=0$$ for any $j>\phi^{t}_{sp, 2}(\mathcal{E})$. This implies that $\phi^{t}_{sp, 1}(\mathcal{E})\leq\phi^{t}_{sp, 2}(\mathcal{E})$. Similarly, one can show that $\phi^{t}_{sp, 2}(\mathcal{E})\leq\phi^{t}_{sp, 1}(\mathcal{E})$. Therefore $\phi^{t}_{sp, 2}(\mathcal{E})=\phi^{t}_{sp, 1}(\mathcal{E})$. This gives the conclusion for Definition \ref{defFt}.

One can prove the conclusion for Definition \ref{defFtGG} by a similar method.
\end{proof}

From the definitions above, one immediately obtains that:
\begin{lemma}\label{lemma3.8}
Let $\mathcal{E}$ be a coherent sheaf on $X$. Then we have
\begin{enumerate}
  \item $\phi^t_{sp}(\mathcal{E})\leq \phi^t_{a}(\mathcal{E})$;
  \item $\mathcal{E}$ is $F^t$-ample $\Rightarrow$ $\mathcal{E}$ is $F^t$-semipositive $\Rightarrow$ $\mathcal{E}$ is $F^tGG$-semipositive;
  \item $\mathcal{E}$ is $F^tGG$-ample $\Rightarrow$ $\mathcal{E}$ is $F^tGG$-semipositive;
  \item $\mathcal{E}$ is $F^tGG$-semipositive $\Rightarrow$ $\mathcal{E}$ is weakly $F^tGG$-semipositive;
  \item $\mathcal{E}$ is (weakly) $F^tGG$-semipositive $\Rightarrow$ so is any quotient sheaf of $\mathcal{E}$.
\end{enumerate}
\end{lemma}

\begin{lemma}\label{lemma3.10}
Let $K$ be an extension of $k$ and $\mathcal{E}$ a coherent sheaf on $X$. Then we have $\phi^t_{sp}(\mathcal{E})=\phi^t_{sp}(\mathcal{E}\otimes_kK)$.
\end{lemma}
\begin{proof}
The proof is similar to that of Lemma \ref{lemma2.8}.
\end{proof}

\begin{definition}
A locally free sheaf $\mathcal{E}$ on $X$ is nef if for any finite morphism $f:C\rightarrow X$ from a smooth irreducible projective curve $C$ to $X$, and any quotient line bundle $\mathcal{Q}$ of $f^*\mathcal{E}$, one has $\deg \mathcal{Q}\geq0$. We say $\mathcal{E}$ is arithmetically nef if there is a thickening $(\widetilde{X}, \widetilde{\mathcal{E}})$ over $\Spec A$ such that $\mathcal{E}_s$ is nef for every closed point $s\in\Spec A$.
\end{definition}

Viehweg introduced the notion of weak positivity as a generalization of nefness of locally free sheaves. We use the following modified version given by Koll\'ar \cite[Notation (vii)]{Kol3}.

\begin{definition}\label{wp}
Assume that $X$ is integral. Let $U$ be an open subscheme of $X$. A coherent sheaf $\mathcal{G}$ on $X$ is said to be weakly positive on $U$ if for an ample divisor $H$ and for a given integer $\alpha>0$ there exist some $\beta>0$ such that $(\Sym^{\alpha\beta}\mathcal{G})^{**}\otimes\mathcal{O}_Y(\beta H)$ is globally generated on $U$. We call $\mathcal{G}$ is weakly positive if it is weakly positive on a certain dense open subscheme of $X$.
\end{definition}

\begin{remark}
Definition \ref{wp} is independent of the choice of $H$ (e.g., \cite[Lemma 2.14]{Vie95}).
\end{remark}

\begin{proposition}\label{nef}
Let $\mathcal{E}$ be a coherent sheaf on $X$.
\begin{enumerate}
  \item If $\mathcal{E}$ is $F^tGG$-semipositive and locally free, then it is nef.
  \item If $\mathcal{E}$ is an arithmetically nef line bundle, then it is $F^t$-semipositive.
  \item If $\mathcal{E}$ is weakly $F^tGG$-semipositive and $X$ is smooth, then $\mathcal{E}$ is weakly positive.
\end{enumerate}
\end{proposition}
\begin{proof}
(1) Let $f:C\rightarrow X$ be a finite morphism from a smooth irreducible projective curve $C$ to $X$ and $\mathcal{Q}$ be a quotient line bundle of $f^*\mathcal{E}$. Let $\mathcal{O}_X(1)$ be a very ample line bundle. Take a thickening $(\widetilde{f}:\widetilde{C}\rightarrow\widetilde{X}, \mathcal{O}_{\widetilde{X}}(1),\widetilde{f}^*\widetilde{\mathcal{E}}\rightarrow\widetilde{\mathcal{Q}})$ over $\Spec A$. If $\mathcal{E}$ is $F^tGG$-semipositive and locally free, there exists an integer $m_0$ such that $\mathcal{E}_s^{(p_s^t)}(m_0)$ is globally generated for almost all $s$. So is $\mathcal{Q}_s^{(p_s^t)}(m_0)$. This implies
\begin{eqnarray*}
  \deg(\mathcal{Q}_s^{(p_s^t)}(m_0))=p_s^t\deg \mathcal{Q}+m_0\deg f^*\mathcal{O}_{X}(1)\geq0
\end{eqnarray*}
for almost all $s$. Taking $p_s\rightarrow\infty$, one infers $\deg\mathcal{Q}\geq0$. Hence $\mathcal{E}$ is nef.

(2) The proof of the second statement is the same as that of \cite[Prop. B.1]{Ara1}.

(3) The proof of the third statement is similar to that of \cite[Prop. 4.7]{Ejiri}. Since a torsion sheaf is always generically globally generated, we can assume $\mathcal{E}$ is torsion free. Let $\dim X=d$. Let $H$ be a very ample divisor on $X$. By definition, there is an integer $m_0$, an open subscheme $U$ of $X$ and a thickening $(\widetilde{U}\subset\widetilde{X}, \widetilde{H}, \widetilde{\mathcal{E}})$ over $\Spec A$ such that $\mathcal{E}$ is locally free on $U$ and $\mathcal{E}_s^{(p_s^t)}(m_0H_s)$ is globally generated on $U_s$ for almost all $s$. For a given positive integer $\alpha$, set $\mathcal{G}:=(\Sym^{\alpha}\mathcal{E})^{**}$. One obtains a morphism
$$\bigoplus\mathcal{O}_{X_s}\rightarrow \mathcal{G}_s^{(p_s^t)}(\alpha m_0H_s)$$
which is surjective on $U_s$ for almost all $s$.

Choose an integer $n$ such that $B:=(n-d-1)H-K_X$ is ample. Then we have the following surjective morphism on $U_s$
\begin{eqnarray*}
 && \bigoplus F_{s, *}^t\mathcal{O}_{X_s}(K_{X_s}+p_s^t(nH_s-K_{X_s})-\alpha m_0H_s)\\
  &\rightarrow &  F_{s, *}^t\mathcal{G}_s^{(p_s^t)}(K_{X_s}+p_s^t(nH_s-K_{X_s}))\\
  &\cong&\mathcal{G}_s\otimes F_{s, *}^t\mathcal{O}_{X_s}(K_{X_s}+p_s^t(nH_s-K_{X_s}))\\
   &\cong&\mathcal{G}_s(nH_s)\otimes F_{s, *}^t\mathcal{O}_{X_s}((1-p^t_s)K_{X_s})\\
   &\rightarrow &\mathcal{G}_s(nH_s),
\end{eqnarray*}
where the isomorphism in fourth and fifth line is induced by the projection formula, and the last morphism is induced by the trace map.

On the other hand, by Serre vanishing, we have
\begin{eqnarray*}
  &&H^i(X_s, \mathcal{O}_{X_s}(-iH_s)\otimes F_{s, *}^t\mathcal{O}_{X_s}(K_{X_s}+p_s^t(nH_s-K_{X_s})-\alpha m_0H_s)) \\
   &\cong& H^i(X_s, F_{s, *}^t\mathcal{O}_{X_s}(K_{X_s}+p_s^t(nH_s-K_{X_s}-iH_s)-\alpha m_0H_s))\\
   &\cong& H^i(X_s, F_{s, *}^t\mathcal{O}_{X_s}(K_{X_s}+p_s^t(B_s+(d+1-i)H_s)-\alpha m_0H_s))\\
   &=&0,
\end{eqnarray*}
for $i>0$ and $p_s\gg0$. Hence $F_{s, *}^t\mathcal{O}_{X_s}(K_{X_s}+p_s^t(nH_s-K_{X_s})-\alpha m_0H_s)$ is globally generated by Lemma \ref{lemma3.3}. This implies that $\mathcal{G}_s(nH_s)$ is globally generated on $U_s$ for almost all $s$. Thus $\mathcal{G}(nH)=(\Sym^{\alpha}\mathcal{E})^{**}(nH)$  is globally generated on $U$. This completes the proof.
\end{proof}

The following examples show that an ample vector bundle can not be weakly $F^tGG$-semipositive, and an $F^tGG$-semipositive sheaf can not be $F^t$-semipositive.
\begin{example}\label{eg3.16}
Like the construction in Example \ref{eg2.13}, there are ample vector bundles $\mathcal{E}_1$ and $\mathcal{E}_3$ on $\mathbb{P}^2$ having presentations of the following forms:
\begin{eqnarray*}
 &&0\rightarrow\mathcal{O}_{\mathbb{P}^2}(-d)^{\oplus2}\rightarrow\mathcal{O}_{\mathbb{P}^2}(-1)^{\oplus4}\rightarrow\mathcal{E}_1\rightarrow0;\\
 &&0\rightarrow\mathcal{O}_{\mathbb{P}^2}(-d)^{\oplus2}\rightarrow\mathcal{O}_{\mathbb{P}^2}^{\oplus4}\rightarrow\mathcal{E}_3\rightarrow0.
\end{eqnarray*}
One notices that $\mathcal{O}_{\mathbb{P}^2}^{\oplus4}$ is $F^tGG$-semipositive, so is $\mathcal{E}_3$. Choose a thickening of these two sequences over $\Spec A$. For a general closed point $s\in \Spec A$, we consider the $t$-iterated Frobenius pull backs of the restrictions of these exact sequences on the fiber $Y_s$. Let $m$ be a given positive integer. The long exact sequences for cohomology give $h^0(Y_s, \mathcal{E}^{(p_s^t)}_{1,s})=0$ and
$$h^1(Y_s, \mathcal{E}^{(p_s^t)}_{3,s}(m))=2h^2(\mathbb{P}^2, \mathcal{O}_{\mathbb{P}^2}(m-dp_s^t))=2h^0(\mathbb{P}^2, \mathcal{O}_{\mathbb{P}^2}(dp_s^t-3-m))>0$$ when $p_s\gg0$. Therefore $\mathcal{E}_1$ is not weakly $F^tGG$-semipositive, and $\mathcal{E}_3$ is not $F^t$-semipositive.
\end{example}

\begin{proposition}[$F^t$-semipositivity of pullbacks]\label{pullback}
Let $f: X\rightarrow Y$ be a morphism of projective schemes over $k$. Let $\mathcal{E}$ be a locally free sheaf on $Y$. Then we have $\phi^t_{sp}(f^*\mathcal{E})\leq \phi^t_{sp}(\mathcal{E})$.
\end{proposition}
\begin{proof}
Let $\mathcal{O}_X(1)$ (resp. $\mathcal{O}_Y(1)$) be a very ample line bundle on $X$ (resp. $Y$). Then there exists a positive integer $m_0$ such that
\begin{equation}\label{3.1}
 R^if_*(f^*\mathcal{E}\otimes\mathcal{O}_X(m-i))=R^if_*\mathcal{O}_X(m-i)\otimes\mathcal{E}=0
\end{equation}
for all $i>0$ and $m\geq m_0$. Choose a thickening
$(\widetilde{f}: \widetilde{X}\rightarrow \widetilde{Y}, \mathcal{O}_{\widetilde{X}}(1), \mathcal{O}_{\widetilde{Y}}(1), \widetilde{\mathcal{E}})$ over $\Spec A$. By semicontinuity one can find an integer $r_0$ such that $$r_0\geq\reg_0(f_{s,*}\mathcal{O}_{X_s}(m_0-i))$$ for $1\leq i\leq \dim X$ and almost all $s\in\Spec A$. From the definition of $\phi^t_{sp}(\mathcal{E})$, there exists an integer $r_1$ such that $\reg_u(\mathcal{E}_s^{(p_s^t)})\leq r_1$ for any $u\geq\phi^t_{sp}(\mathcal{E})$ and almost all $s$. Then by Proposition \ref{tensor-CM} we can take an integer $r_2$ such that
$$\reg_u(\mathcal{E}_s^{(p_s^t)}\otimes f_{s,*}\mathcal{O}_{X_s}(m_0-i))\leq r_2$$ for any $1\leq i\leq \dim X$, $u\geq\phi^t_{sp}(\mathcal{E})$ and almost all $s\in\Spec A$. Using the commutative diagram
$$\xymatrix{
  X_s  \ar[d]_{f_s} \ar[r]^{F^t_{X_s}}  & X_s \ar[d]^{f_s}            \\
 Y_s \ar[r]^{F^t_{Y_s}} & Y_s                       }$$
and (\ref{3.1}), one obtains
\begin{eqnarray*}
  &&H^i(X_s, (f_s^*\mathcal{E}_s)^{(p_s^t)}\otimes\mathcal{O}_{X_s}(m_0-i)\otimes f_s^*\mathcal{O}_{Y_s}(r-i)) \\
  &&=H^i(Y_s, \mathcal{E}_s^{(p_s^t)}\otimes f_{s, *}\mathcal{O}_{X_s}(m_0-i)\otimes\mathcal{O}_{Y_s}(r-i)) \\
 &&= 0
\end{eqnarray*}
for any $i>\phi^t_{sp}(\mathcal{E})$, $r\geq r_2$ and almost all $s$. In particular, $$H^i(X_s, (f_s^*\mathcal{E}_s)^{(p_s^t)}\otimes\mathcal{O}_{X_s}(m_0-i)\otimes f_s^*\mathcal{O}_{Y_s}(r))=0$$ when $r\geq r_2$. Thus
$\reg_u((f_s^*\mathcal{E}_s)^{(p_s^t)}\otimes f_s^*\mathcal{O}_{Y_s}(r_2))\leq m_0$ for $u\geq\phi^t_{sp}(\mathcal{E})$ and almost all $s$.
Using Proposition \ref{tensor-CM} and semicontinuity again, one obtains
\begin{eqnarray*}
  \reg_u((f_s^*\mathcal{E}_s)^{(p_s^t)}) &=& \reg_u((f_s^*\mathcal{E}_s)^{(p_s^t)}\otimes f_s^*\mathcal{O}_{Y_s}(r_2)\otimes f_s^*\mathcal{O}_{Y_s}(-r_2)) \\
   &\leq& m_0+\reg_0(f_s^*\mathcal{O}_{Y_s}(-r_2))+(\dim X-u-1)(R-1) \\
   &\leq& m_0+\reg_0(f^*\mathcal{O}_{Y}(-r_2))+ (\dim X-1)(R-1)
\end{eqnarray*}
for $u\geq\phi^t_{sp}(\mathcal{E})$ and almost all $s$, where $R=\max\{1, \reg_0(\mathcal{O}_X)\}$. This implies that $\reg_u((f_s^*\mathcal{E}_s)^{(p_s^t)})$ are bounded for any $u\geq\phi^t_{sp}(\mathcal{E})$ and almost all $s$. Hence the conclusion follows.
\end{proof}

\begin{proposition}[$F^t$-semipositivity of restrictions]\label{restriction}
Let $\mathcal{O}_X(1)$ be a very ample line bundle and $\mathcal{E}$ a coherent sheaf on $X$. Assume either that $\mathcal{E}$ is locally free or that $X$ is smooth and $k$ is infinite. Then for a general divisor $H$ in the linear system $|\mathcal{O}_X(1)|$ we have $$\mathcal{\phi}^t_{sp}(\mathcal{E}|_H)\leq\mathcal{\phi}^t_{sp}(\mathcal{E})\leq\mathcal{\phi}^t_{sp}(\mathcal{E}|_H)+1.$$
\end{proposition}
\begin{proof}
Take a thickening $(\widetilde{X}, \mathcal{O}_{\widetilde{X}}(1), \widetilde{H}, \widetilde{\mathcal{E}})$ over $\Spec A$ and let $s\in \Spec A$ be a general closed point. In the case that $X$ is smooth and $k$ is infinite, from \cite[Lemma 1.1.12]{HL}, one sees that for a general section $u\in H^0(X, \mathcal{O}_X(1))$ the natural morphism
$$\theta: \mathcal{E}(-1)\xrightarrow{\cdot u} \mathcal{E}$$ is injective. The smoothness of $X$ implies that $X_s$ is smooth and its Frobenius morphism $F_{s}$ is flat. Thus the morphism $$(F^t_{s})^*\theta_s: \mathcal{E}_s^{(p_s^t)}(-p_s^t)\xrightarrow{\cdot u_s^{p_s^t}}\mathcal{E}_s^{(p_s^t)}$$ is still injective. Consider the morphism $$\tau_{i, s}: \mathcal{E}_s^{(p_s^t)}(-i)\xrightarrow{\cdot u_s} \mathcal{E}_s^{(p_s^t)}(-i+1).$$ It turns out that $(F^t_{s})^*\theta_s$ is the composition of the $\tau_{i, s}$'s, here $1\leq i\leq p_s^t$. This implies that $\tau_{i, s}$ is injective for any $i$. In the case that $\mathcal{E}$ is locally free, the injectivity of $\tau_{i, s}$ is obvious. Therefore, in any case one obtains an exact sequence

\begin{equation}\label{3.2}
0\rightarrow\mathcal{E}_s^{(p_s^t)}(m-1)\rightarrow\mathcal{E}_s^{(p_s^t)}(m)\rightarrow\mathcal{E}_s^{(p_s^t)}(m)|_{H_s}\rightarrow0
\end{equation}
for any integer $m$.

By the definition of $\phi^t_{sp}(\mathcal{E})$, one sees that there exists an integer $m_0$ such that $$H^i(X_s, \mathcal{E}_s^{(p_s^t)}(m-i))=0$$
for all $i>\phi^t_{sp}(\mathcal{E})$, $m\geq m_0$ and general $s$. On the other hand, from the short exact sequence (\ref{3.2}), we obtain the induced long exact sequence
$$H^i(X_s, \mathcal{E}_s^{(p_s^t)}(m-i))\rightarrow H^i(H_s, \mathcal{E}_s^{(p_s^t)}(m-i)|_{H_s})\rightarrow H^{i+1}(X_s, \mathcal{E}_s^{(p_s^t)}(m-i-1)).$$
It follows that $H^i(H_s, \mathcal{E}_s^{(p_s^t)}(m-i)|_{H_s})=0$ for $i>\phi^t_{sp}(\mathcal{E})$ and $m\geq m_0$. Thus we infer that $\mathcal{\phi}^t_{sp}(\mathcal{E}|_{H})\leq\mathcal{\phi}^t_{sp}(\mathcal{E})$.

Similarly, from the definition of $\phi^t_{sp}(\mathcal{E}|_H)$, we deduce that there exists an integer $m_1$ such that $$H^i(H_s, \mathcal{E}_s^{(p_s^t)}(m-i)|_{H_s})=0$$
for all $i>\phi^t_{sp}(\mathcal{E}|_H)$ and $m\geq m_1$. The short exact sequence (\ref{3.1}) gives the long exact sequence
\begin{eqnarray*}
&&H^i(H_s, \mathcal{E}_s^{(p_s^t)}(m-i)|_{H_s})\rightarrow H^{i+1}(X_s, \mathcal{E}_s^{(p_s^t)}(m-i-1))\\
&&\rightarrow H^{i+1}(X_s, \mathcal{E}_s^{(p_s^t)}(m-i))\rightarrow H^{i+1}(H_s, \mathcal{E}_s^{(p_s^t)}(m-i)|_{H_s}).
\end{eqnarray*}
This implies that $$H^{i+1}(X_s, \mathcal{E}_s^{(p_s^t)}(m-i-1))\cong H^{i+1}(X_s, \mathcal{E}_s^{(p_s^t)}(m-i))$$ for any $i>\phi^t_{sp}(\mathcal{E}|_H)$ and $m\geq m_1$. Hence one sees that $$H^{i+1}(X_s, \mathcal{E}_s^{(p_s^t)}(m-i-1))\cong H^{i+1}(X_s, \mathcal{E}_s^{(p_s^t)}(m-i+b))$$ for any $i>\phi^t_{sp}(\mathcal{E}|_H)$, $m\geq m_1$ and $b\geq0$. By Serre's vanishing, one has $$H^{i+1}(X_s, \mathcal{E}_s^{(p_s^t)}(m-i+b))=0$$ for $b\gg0$. Therefore, we infer that $$H^{i+1}(X_s, \mathcal{E}_s^{(p_s^t)}(m-i-1))=0$$ for any $i>\phi^t_{sp}(\mathcal{E}|_{H})$ and $m\geq m_1$.
It follows that
$\mathcal{\phi}^t_{sp}(\mathcal{E})\leq\mathcal{\phi}^t_{sp}(\mathcal{E}|_H)+1$.
\end{proof}

\begin{lemma}\label{pullback2}
Let $\pi: Y\rightarrow Z$ be a finite morphism between smooth quasi-projective schemes over $k$. Let $\mathcal{E}$ be a coherent sheaf on $Z$. Then $\mathcal{E}$ is (weakly) $F^tGG$-semipositive if and only if $\pi^*\mathcal{E}$ is (weakly) $F^tGG$-semipositive.
\end{lemma}
\begin{proof}
We only show the conclusion for the $F^tGG$-semipositive case. The weakly $F^tGG$-semipositive case can be proved in the same way.

Let $\mathcal{O}_Z(1)$ be a very ample line bundle on $Z$ such that $\pi^*\mathcal{O}_Z(1)$ is still very ample. Choose a thickening $$(\widetilde{\pi}: \widetilde{Y}\rightarrow \widetilde{Z}, \mathcal{O}_{\widetilde{Z}}(1), \widetilde{\mathcal{E}})$$ over $\Spec A$ and a general closed point $s\in \Spec A$. We have the commutative diagram
$$\xymatrix{
  Y_s  \ar[d]_{\pi_s} \ar[r]^{F^t_{Y_s}}  & Y_s \ar[d]^{\pi_s}            \\
 Z_s \ar[r]^{F^t_{Z_s}} & Z_s.                       }$$

Suppose that $\mathcal{E}$ is $F^tGG$-semipositive. Then one can find an integer $m_0$ such that $(F^t_{Z_s})^*\mathcal{E}_s\otimes\mathcal{O}_{Z_s}(m_0)$ is globally generated for almost all $s$. So is $$\pi_s^*((F^t_{Z_s})^*\mathcal{E}_s\otimes\mathcal{O}_{Z_s}(m_0))\cong (F^t_{Y_s})^*\pi_s^*\mathcal{E}_s\otimes\pi_s^*\mathcal{O}_{Z_s}(m_0).$$ Thus $\pi^*\mathcal{E}$ is $F^tGG$-semipositive.

For the other direction, we suppose that $\pi^*\mathcal{E}$ is $F^tGG$-semipositive. Then there is surjective morphism
$$\bigoplus\mathcal{O}_{Y_s}\rightarrow (F^t_{Y_s})^*\pi_s^*\mathcal{E}_s\otimes\pi_s^*\mathcal{O}_{Z_s}(m)$$ for almost all $s$.
Since the trace map splits the natural inclusion $\mathcal{O}_{Z_s}\rightarrow\pi_{s,*}\mathcal{O}_{Y_s}$ when $\chr(A/s)\nmid \deg\pi_s$, the induced morphism
\begin{eqnarray*}
\bigoplus\pi_{s, *}\mathcal{O}_{Y_s}&\rightarrow&\pi_{s,*}((F^t_{Y_s})^*\pi_s^*\mathcal{E}_s\otimes\pi_s^*\mathcal{O}_{Z_s}(m))\\
&\cong&(F^t_{Y_s})^*\mathcal{E}_s\otimes\mathcal{O}_{Z_s}(m)\otimes\pi_{s, *}\mathcal{O}_{Y_s}\\
&\rightarrow&(F^t_{Y_s})^*\mathcal{E}_s\otimes\mathcal{O}_{Z_s}(m)
\end{eqnarray*}
is surjective for almost all $s$, here the isomorphism follows from \cite[Lemma 5.6]{Ara1}.
Taking an integer $m_1$ such that $\pi_*\mathcal{O}_Y\otimes\mathcal{O}_Z(m_1)$ is  globally generated, one sees that $\pi_{s, *}\mathcal{O}_{Y_s}\otimes\mathcal{O}_{Z_s}(m_1)$ is globally generated for almost all $s$. So is $$(F^t_{Y_s})^*\mathcal{E}_s\otimes\mathcal{O}_{Z_s}(m+m_1).$$ Therefore $\mathcal{E}$ is $F^tGG$-semipositive.
\end{proof}

The $F^t$-semipositivity can be governed by Castelnuovo–Mumford regularity:
\begin{proposition}\label{est-CM}
Let $\mathcal{O}_X(1)$ be a very ample line bundle and $\mathcal{E}$ an $m$-regular sheaf on $X$. Assume either that $X$ is smooth with $\dim X=d$ or that $\mathcal{E}$ is locally free. Then we have
$$\phi^t_{sp}(\mathcal{E}(m+uR))\leq\max\{d-u-1, 0\},$$ where $R=\max\{1, \reg_0(\mathcal{O}_X)\}$.
\end{proposition}
\begin{proof}
Choosing a thickening $(\widetilde{X}, \mathcal{O}_{\widetilde{X}}(1), \widetilde{\mathcal{E}})$ over $\Spec A$ and a general closed point $s\in \Spec A$, by semicontinuity and Lemma \ref{lemma3.3} we obtain the exact sequence
$$\cdots\rightarrow\mathcal{E}_{d, s}\rightarrow\cdots\rightarrow\mathcal{E}_{0, s}\rightarrow\mathcal{E}_s(m+uR)\rightarrow0,$$ where $\mathcal{E}_{i, s}=V_i\otimes\mathcal{O}_{X_s}((u-i)R)$.
Pulling back it via $(F_{s}^t)^*$, one gets the exact complex
$$\cdots\rightarrow V_u\otimes\mathcal{O}_{X_s}\rightarrow\cdots\rightarrow V_0\otimes\mathcal{O}_{X_s}(p_s^tuR)\rightarrow(\mathcal{E}_s(m+uR))^{(p_s^t)}\rightarrow0.$$
When $i\geq d-u$, one deduced that
$$h^i\Big((\mathcal{E}_s(m+uR))^{(p_s^t)}(R-i)\Big)\leq\sum_{j=0}^{u}h^{i+j}\Big(V_j\otimes\mathcal{O}_{X_s}(p_s^t(u-j)R+R-i)\Big)=0.$$
This completes the proof.
\end{proof}

The following theorem is a variant of \cite[Theorem 3]{Ara1} and \cite[Theorem 4.8]{Sun}.

\begin{theorem}\label{tensor}
Let $\mathcal{E}$ and $\mathcal{F}$ be two coherent sheaves on $X$ such that one of them is locally free and $\mathcal{E}$ is $F^t$-ample, then
$$\phi^t_a(\mathcal{E}\otimes\mathcal{F})\leq\phi^t_{sp}(\mathcal{F}).$$
\end{theorem}
\begin{proof}
Let $\mathcal{O}_X(1)$ be a very ample line bundle on $X$ and $(\widetilde{X}, \widetilde{\mathcal{E}}, \widetilde{\mathcal{F}}, \mathcal{O}_{\widetilde{X}}(1))$ be a thickening of $(X, \mathcal{E}, \mathcal{F}, \mathcal{O}_X(1))$ over $\Spec A$.

By the definition of $\phi^t_{sp}(\mathcal{F})$, one sees that there is an integer $m_0$ such that
$$H^i(X_s, \mathcal{F}_s^{(p_s^t)}(m-i))=0$$ for all $i>\phi^t_{sp}(\mathcal{F})$, $m\geq m_0$ and general $s$. Let $d=\dim X$ and $R=\max\{1, \reg_0(\mathcal{O}_X)\}$.
Since $\mathcal{E}$ is $F^t$-ample, one sees that for any integer $b$ there is a nonempty open subset $U_b$ of $\Spec A$ such that $$\reg_0(\mathcal{E}_s^{(p_s^t)})\leq -m_0-dR+b$$ for any closed point $s\in U_b$. From Lemma \ref{lemma3.3}, we obtain an exact sequence $$\cdots\rightarrow\mathcal{E}_{d, s}\rightarrow\cdots\rightarrow\mathcal{E}_{0, s}\rightarrow\mathcal{E}_s^{(p_s^t)}\rightarrow0,$$
where $\mathcal{E}_{i, s}=V_i\otimes\mathcal{O}_{X_s}(m_0+dR-iR-b)$. Twisting through by $\mathcal{F}_s^{(p_s^t)}(b)$ yields an exact complex
$$\cdots\rightarrow\mathcal{E}_{1,s}\otimes\mathcal{F}_s^{(p_s^t)}(b)\rightarrow\mathcal{E}_{0, s}\otimes\mathcal{F}_s^{(p_s^t)}(b)
\rightarrow\mathcal{E}_s^{(p^t)}\otimes\mathcal{F}_s^{(p_s^t)}(b)\rightarrow0.$$
By chasing through this exact complex, it follows that
$$H^i(X_s, (\mathcal{E}_s\otimes\mathcal{F}_s)^{(p_s^t)}\otimes\mathcal{O}_{X_s}(b))=0$$ for any $i>\phi^t_{sp}(\mathcal{F})$ and closed point $s\in U_b$. Hence the conclusion follows from Lemma \ref{lemma2.6}.
\end{proof}

\begin{theorem}\label{tensor2}
Assume that $X$ is smooth. Let $\mathcal{E}$ and $\mathcal{F}$ be two locally free sheaves on $X$, then we have
\begin{enumerate}
  \item $\phi^t_{sp}(\mathcal{E}\otimes\mathcal{F})\leq\phi^t_{sp}(\mathcal{E})+\phi^t_{sp}(\mathcal{F})$;
  \item $\phi^t_{a}(\mathcal{E}\otimes\mathcal{F})\leq\phi^t_{a}(\mathcal{E})+\phi^t_{a}(\mathcal{F})$.
\end{enumerate}
\end{theorem}
\begin{proof}
We only give the proof of the first inequality. The proof of the second is similar to that of the first.

Let $Y=X\times X$ and let $p_i$ denote the projections. Given two coherent sheaves $\mathcal{E}_1$ and $\mathcal{E}_2$ on $X$, we write $\mathcal{E}_1\boxtimes\mathcal{E}_1=p_1^*\mathcal{E}_1\otimes p_2^*\mathcal{E}_2$. Choose a very ample line bundle $\mathcal{O}_X(1)$ and a thickening $(\widetilde{X}, \mathcal{O}_{\widetilde{X}}(1), \widetilde{\mathcal{E}}, \widetilde{\mathcal{F}})$ over $\Spec A$.
Let $\widetilde{\Delta}\subset \widetilde{Y}:=\widetilde{X}\times\widetilde{X}$ be the diagonal and $\widetilde{\mathcal{L}}=\mathcal{O}_{\widetilde{X}}(1)\boxtimes\mathcal{O}_{\widetilde{X}}(1)$.

By semicontinuity and Lemma \ref{lemma3.3}, we obtain the exact sequence
\begin{equation}\label{diag}
 \cdots\rightarrow\mathcal{G}_{d, s}\rightarrow\cdots\rightarrow\mathcal{G}_{0, s}\rightarrow\mathcal{O}_{\Delta_s}\rightarrow0
\end{equation}
for almost all $s\in\Spec A$, where $\mathcal{G}_{i, s}=V_i\otimes \mathcal{L}^{b_i}_s$ and $d=\dim X$. Hence using K\"unneth's formula and the definition of $F^t$-semipositivity, one sees that there exists an integer $N$ such that
\begin{eqnarray*}
&& H^i(\mathcal{G}_{j, s}\otimes (\mathcal{E}_s\boxtimes\mathcal{F}_s)^{(p_s^t)}\otimes\mathcal{L}^{m}_s) \\
&=&V_j\otimes H^i(\mathcal{E}_s^{(p_s^t)}(m+b_j)\boxtimes\mathcal{F}_s^{(p_s^t)}(m+b_j)) \\
&=&V_j\otimes\bigoplus_{l+n=i}H^l(\mathcal{E}_s^{(p_s^t)}(m+b_j))\otimes H^n(\mathcal{F}_s^{(p_s^t)}(m+b_j))\\
&=&0
\end{eqnarray*}
for any $i>\phi^t_{sp}(\mathcal{E})+\phi^t_{sp}(\mathcal{F})$, $j\leq d$, $m\geq N$ and almost all $s$. Tensoring (\ref{diag}) by $(\mathcal{E}_s\boxtimes\mathcal{F}_s)^{(p_s^t)}\otimes\mathcal{L}^{m}_s$ and chasing through the exact complex, one infers that
$$H^i((\mathcal{E}_s\otimes\mathcal{F}_s)^{(p_s^t)}(2m))=H^i(\mathcal{O}_{\Delta_s}\otimes
(\mathcal{E}_s\boxtimes\mathcal{F}_s)^{(p_s^t)}\otimes\mathcal{L}^{m}_s)=0$$ for $i>\phi^t_{sp}(\mathcal{E})+\phi^t_{sp}(\mathcal{F})$, $m\geq N$ and almost all $s$. Thus Remark \ref{rm3.6} gives the desired inequality.
\end{proof}

Since any symmetric or exterior power of a locally free sheaf is a direct summand of some tensor power of the sheaf, we have the following:
\begin{corollary}\label{exterior}
Any tensor, symmetric or exterior power of an $F^t$-ample ($F^t$-semipositive) locally free sheaf is $F^t$-ample ($F^t$-semipositive).
\end{corollary}

\section{$F^t$-semipositivity and vanishing theorems}\label{S4}
In this section, we give some generalizations of Fujita's vanishing theorem and the Kawamata-Viehweg vanishing theorem. We let $X$ be a projective scheme over a field $k$ with $\chr(k)=0$ throughout this section. We keep the same notations as that in the previous sections.

The following theorem is a variant of Fujita type vanishing theorem proved by Keeler \cite{Ke}.

\begin{theorem}\label{Fujita1}
Assume that $X$ is smooth. Let $\mathcal{L}$ be an ample line bundle on $X$. Given any locally free sheaf $\mathcal{F}$ on $X$, there exists an integer $m(\mathcal{F}, \mathcal{L})$ such that
$$H^i(X, \mathcal{F}\otimes\mathcal{L}^m\otimes\mathcal{G})=0$$ for all $i>u$, $m\geq m(\mathcal{F}, \mathcal{L})$ and any coherent sheaf $\mathcal{G}$ on $X$ with $\phi^t_{sp}(\mathcal{G})\leq u$.
\end{theorem}
\begin{proof}
The proof is similar to that of \cite[Theorem 4.5]{Ke}. Let $\mathcal{O}_X(1)$ be a very ample line bundle on $X$. Let $d=\dim X$ and $R=\max\{1, \reg_0(\mathcal{O}_X)\}$. Assume that $\phi^t_{sp}(\mathcal{G})\leq u$. By Theorem \ref{tensor}, one sees that $\phi^t_{a}(\mathcal{L}^m\otimes\mathcal{G})\leq\phi^t_{sp}(\mathcal{G})$ for any $m\geq1$. This implies
$$H^j(X, \omega_X\otimes\mathcal{L}^m\otimes\mathcal{G})=0$$ for any $j>u$ and $m\geq1$ by Theorem \ref{thm2.13}. Hence one can take $m_1$ sufficiently large such that $$\reg_u(\omega_X\otimes\mathcal{L}^m\otimes\mathcal{G})<(u+1-d)(R-1)$$ for $m\geq m_1$. Taking $m_2$ sufficiently large, we can make $\reg_0(\mathcal{F}\otimes\omega^*_X\otimes\mathcal{L}^{m_2})<0$. Then by Proposition \ref{tensor-CM}, we have $\reg_u(\mathcal{F}\otimes\mathcal{L}^m\otimes\mathcal{G})<0$ for $m\geq m_1+m_2$. This immediately gives the theorem.
\end{proof}

We may have another generalization of Fujita's vanishing theorem without smoothness assumption.
\begin{theorem}\label{Fujita2}
Let $\mathcal{L}$ be an ample line bundle on $X$. Given any coherent sheaf $\mathcal{F}$ on $X$, there exists an integer $m(\mathcal{F}, \mathcal{L})$ such that
$$H^i(X, \mathcal{F}\otimes\mathcal{L}^m\otimes\mathcal{G})=0$$ for all $i>u$, $m\geq m(\mathcal{F}, \mathcal{L})$ and any locally free sheaf $\mathcal{G}$ on $X$ with $\phi^t_{sp}(\mathcal{G})\leq u$.
\end{theorem}
\begin{proof}
By Lemma \ref{lemma3.10}, we may assume that $k$ is algebraically closed. We use the induction on dimension. When $\dim X=0$ the theorem obviously holds.
\smallskip

\noindent\textbf{Step 1}. We can reduce the proof to the case that $X$ is reduced.

\smallskip
We assume that the theorem holds for reduced schemes. Let $f: X_{red}\rightarrow X$ be the natural immersion.
Let $\mathcal{N}$ be the nilradical of $\mathcal{O}_X$, so that $\mathcal{N}^r=0$ for some positive integer $r$. Consider the
filtration $$\mathcal{F}\supset\mathcal{N}\cdot\mathcal{F}\supset\mathcal{N}^2\cdot
\mathcal{F}\supset\cdots\supset\mathcal{N}^r\cdot\mathcal{F}=0.$$ The quotients $\mathcal{N}^j\mathcal{F}/\mathcal{N}^{j+1}\mathcal{F}$ are coherent $\mathcal{O}_{X_{red}}$-modules. Then there exist coherent sheaves $\mathcal{M}_j$ on $X_{red}$ such that $f_*\mathcal{M}_j=\mathcal{N}^j\mathcal{F}/\mathcal{N}^{j+1}\mathcal{F}$. By Proposition \ref{pullback}, one sees that
$$\phi^t_{sp}(f^*\mathcal{G})\leq\phi^t_{sp}(\mathcal{G})\leq u.$$ Therefore,
by the assumption,
$$H^i(X, (\mathcal{N}^j\mathcal{F}/\mathcal{N}^{j+1}\mathcal{F})\otimes\mathcal{L}^m\otimes\mathcal{G})=H^i(X_{red}, \mathcal{M}_j\otimes f^*\mathcal{L}^m\otimes f^*\mathcal{G})=0$$ for $i>u$ and $m\geq m(\mathcal{N}^j\mathcal{F}/\mathcal{N}^{j+1}\mathcal{F}, \mathcal{L})$.
From the exact sequences $$0\rightarrow\mathcal{N}^{j+1}\mathcal{F}\rightarrow\mathcal{N}^j\mathcal{F}\rightarrow
\mathcal{N}^j\mathcal{F}/\mathcal{N}^{j+1}\mathcal{F}\rightarrow0,$$ it follows that $$H^i(X, (\mathcal{N}^j\mathcal{F}\otimes\mathcal{L}^m\otimes\mathcal{G})=0$$ for any $i>u$ , $r\geq j\geq0$ and $m\geq m(\mathcal{N}^j\mathcal{F},\mathcal{L})$. When $j=0$ we obtain the desired
vanishing.

\smallskip

\noindent\textbf{Step 2}. We can reduce the proof to the case that $X$ is irreducible.

\smallskip
Assume that the theorem holds for reduced and irreducible schemes. Let $X=X_1\cup\cdots\cup X_n$ be its decomposition into irreducible
components and $\mathcal{I}_1$ be the ideal sheaf of $X_1$ in $X$. We consider the exact sequence
$$0\rightarrow\mathcal{I}_1\cdot\mathcal{F}\rightarrow\mathcal{F}\rightarrow\mathcal{F}/\mathcal{I}_1\cdot\mathcal{F}\rightarrow0.$$
Since $\mathcal{I}_1\mathcal{F}$ and $\mathcal{F}/\mathcal{I}_1\mathcal{F}$ are supported on $X_2\cup\cdots\cup X_n$ and $X_1$ respectively, by induction on $n$ and Proposition \ref{pullback}, we may assume that
$$H^i(X, (\mathcal{I}_1\mathcal{F})\otimes\mathcal{L}^m\otimes\mathcal{G})=H^i(X, (\mathcal{F}/\mathcal{I}_1\mathcal{F})\otimes\mathcal{L}^m\otimes\mathcal{G})=0$$ for $i>u$ and
$m\geq m(\mathcal{I}_1\mathcal{F}, \mathcal{F}/\mathcal{I}_1\mathcal{F}, H)$. The induced long exact sequence for cohomology of the above
exact sequence shows the required vanishing.

\smallskip

\noindent\textbf{Step 3}. We show Theorem \ref{Fujita2} for integral schemes.

\smallskip

We assume that $X$ is reduced and irreducible. Let $f: Y\rightarrow X$ be a resolution of singularities and $\mathcal{H}$ an ample line bundle on $Y$. We can take an integer $l_0$ such that $R^if_*\mathcal{H}^l=0$ for any $i>0$ and $l\geq l_0$. By Proposition \ref{pullback} and Theorem \ref{tensor}, one sees that
$$\phi^t_{sp}(f^*\mathcal{L}^m\otimes f^*\mathcal{G})\leq \phi^t_{sp}(\mathcal{L}^m\otimes\mathcal{G})\leq\phi^t_{sp}(\mathcal{G})$$ for any $m\geq0$.
From Theorem \ref{Fujita1}, it follows that
there exists $l_1\geq l_0$ such that $$H^j(Y, \mathcal{H}^l\otimes f^*(\mathcal{L}^m\otimes\mathcal{G}))=H^j(X, f_*\mathcal{H}^l\otimes\mathcal{L}^m\otimes\mathcal{G})=0$$ for any $j>u$, $m\geq0$, $l\geq l_1$ and locally free sheaf $\mathcal{G}$ with $\phi^t_{sp}(\mathcal{G})\leq u$.

Take an $n_0$ such that $(f_*\mathcal{H}^{l_1})^*\otimes\mathcal{L}^{n_0}$ is generated by global sections. Since $f_*\mathcal{H}^{l_1}$ is a rank one torsion free sheaf, one obtains an exact sequence
\begin{equation}\label{4.1}
0\rightarrow f_*\mathcal{H}^{l_1}\rightarrow \mathcal{L}^{n_0}\rightarrow\mathcal{Q}\rightarrow0.
\end{equation}
Note that $\mathcal{Q}$ is a torsion sheaf. By the induction and Proposition \ref{pullback}, one can choose an $m_0$ such that $H^j(\mathcal{Q}\otimes \mathcal{L}^m\otimes\mathcal{G})=0$ for any $j>u$, $m\geq m_0$ and locally free sheaf $\mathcal{G}$ with $\phi^t_{sp}(\mathcal{G})\leq u$. Hence the induced long exact sequence of (\ref{4.1}) gives $H^j(\mathcal{L}^m\otimes\mathcal{G})=0$ for $j>u$, $m\geq m_0+n_0$ and $\phi^t_{sp}(\mathcal{G})\leq u$. By Lemma \ref{lemma3.3}, we may find a resolution of $\mathcal{F}$ by direct sums of $\mathcal{L}^{-j}$. The desired vanishing follows by chasing through the resolution.

We finish the proof of the theorem.
\end{proof}

The following Kawamata-Viehweg type vanishing theorem is a generalization of Arapura-Matsuki-Patel-W\l odarczyk's vanishing theorem \cite[Theorem 1.3]{AMPW}.
\begin{theorem}\label{KV}
Assume that $X$ is smooth with $\dim X=d$ and $k$ is algebraically closed. Let $D=\sum_{j=1}^{l} D_j$ be a simple normal crossings divisor, $A$ an ample $\mathbb{Q}$-divisor and $\mathcal{E}$ a locally free sheaf on $X$. Suppose that $\lceil A\rceil-A\leq D$, and let $G$ be an integral divisor such that $\lceil A\rceil-A\leq G\leq D$.
Then we have $$H^i(X, \Omega^j_X(\log D)(\lceil A\rceil-G)\otimes \mathcal{E})=0$$ for
$i+j>d+\phi_{sp}^{t}(\mathcal{E})$.
\end{theorem}
\begin{proof}
We follow the strategy of the first proof of \cite[Theorem 1.3]{AMPW}.
\smallskip

\noindent\textbf{Step 1}. We prove Theorem \ref{KV} in the setting that $A$ is integral.

\smallskip

By Theorem \ref{tensor}, one sees that $\phi^t_{a}(\mathcal{E}(A))\leq\phi^t_{sp}(\mathcal{E})$. Thus the statement in the case of $G=D$ is nothing but Theorem \ref{thm2.13}. In the case that $$G\leq D^{\prime}:=D-D_1<D,$$ we proceed via induction on the number of the components in $D$ and the dimension of $X$. From the short exact sequence
\begin{eqnarray*}
0\rightarrow\Omega^j_X(\log D^{\prime})(A-G)\otimes \mathcal{E}\rightarrow\Omega^j_X(\log D)(A-G)\otimes \mathcal{E}\\
\rightarrow\Omega^{j-1}_{D_1}(\log (D^{\prime}|_{D_1}))((A-G)|_{D_1})\otimes \mathcal{E}|_{D_1}\rightarrow0,
\end{eqnarray*}
one obtains a long exact sequence in cohomology
\begin{eqnarray*}
\cdots\rightarrow H^i(X, \Omega^j_X(\log D^{\prime})(A-G)\otimes \mathcal{E})\rightarrow H^i(X, \Omega^j_X(\log D)(A-G)\otimes \mathcal{E})\\
\rightarrow H^i(D_1,\Omega^{j-1}_{D_1}(\log (D^{\prime}|_{D_1}))((A-G)|_{D_1})\otimes \mathcal{E}|_{D_1})\rightarrow\cdots.
\end{eqnarray*}
By Theorem \ref{pullback}, one deduces that $\phi_{sp}^{t}(\mathcal{E}|_{D_1})\leq \phi_{sp}^{t}(\mathcal{E})$. Therefore, if $i+j>d+\phi_{sp}^{t}(\mathcal{E})$, by induction one sees that both the first and the last term are 0. These imply that
$$H^i(X, \Omega^j_X(\log D)(A-G)\otimes \mathcal{E})=0$$ if
$i+j>d+\phi_{sp}^{t}(\mathcal{E})$.

\smallskip

\noindent\textbf{Step 2}. We prove Theorem \ref{KV} in the case that $A$ is fractional.

\smallskip
Using Kawamata's covering lemma (\cite[Lemma 3.2.1]{AMPW}) and \cite[Lemma 3.2.2]{AMPW}, one sees that there exist a smooth projective
variety $Y$ and a finite Galois covering $\pi: Y\rightarrow X$ with Galois group $\Gamma:=\Gal(K(Y)/K(X))$ such that $\pi^*A$ is integer, $D_Y$ has simple normal crossings and $$\left[\pi_*\Big(\Omega^j_Y(\log D_Y)(\pi^*A-G_Y)\Big)\right]^{\Gamma}=\Omega^j_X(\log D)(\lceil A\rceil-G).$$
Here $G_Y$ (resp. $D_Y$) is the pull-back of the divisor $G$ (resp. $D$) to $Y$ with the reduced structure. This implies
\begin{eqnarray*}
 \left[\pi_*\Big(\Omega^j_Y(\log D_Y)(\pi^*A-G_Y)\otimes\pi^*\mathcal{E}\Big)\right]^{\Gamma}&=&\left[\pi_*\Big(\Omega^j_Y(\log D_Y)(\pi^*A-G_Y)\Big)\right]^{\Gamma}\otimes\mathcal{E} \\
   &=&\Omega^j_X(\log D)(\lceil A\rceil-G)\otimes\mathcal{E}.
\end{eqnarray*}
Hence, from the consequence of Step 1, it follows that
\begin{eqnarray*}
  H^i(\Omega^j_X(\log D)(\lceil A\rceil-G)\otimes \mathcal{E}) &=& H^i(\Big[\pi_*\Big(\Omega^j_Y(\log D_Y)(\pi^*A-G_Y)\otimes\pi^*\mathcal{E}\Big)\Big]^{\Gamma}) \\
   &=& H^i(\pi_*\Big(\Omega^j_Y(\log D_Y)(\pi^*A-G_Y)\otimes\pi^*\mathcal{E}\Big))^{\Gamma}\\
   &=&H^i(\Omega^j_Y(\log D_Y)(\pi^*A-G_Y)\otimes\pi^*\mathcal{E})^{\Gamma}\\
   &=&0
\end{eqnarray*}
when $i+j>d+\phi_{sp}^{t}(\pi^*\mathcal{E})$.
Since $\phi_{sp}^{t}(\pi^*\mathcal{E})\leq\phi_{sp}^{t}(\mathcal{E})$ by Theorem \ref{pullback}, we obtain the desired conclusion.
\end{proof}

By Lemma \ref{nef}, one sees that $\mathcal{O}_X$ is $F^t$-semipositive. Hence taking $\mathcal{E}=\mathcal{O}_X$,
we recover the vanishing theorem of Arapura-Matsuki-Patel-W\l odarczyk:
\begin{theorem}\label{AMPW}
Assume that $X$ is smooth with $\dim X=d$ and $k$ is algebraically closed. Let $D=\sum_{j=1}^{l} D_j$ be a simple normal crossings divisor, $A$ an ample $\mathbb{Q}$-divisor on $X$. Suppose that $\lceil A\rceil-A\leq D$, and let $G$ be an integral divisor such that $\lceil A\rceil-A\leq G\leq D$.
Then we have $$H^i(X, \Omega^j_X(\log D)(\lceil A\rceil-G))=0$$ for
$i+j>d$.
\end{theorem}

\section{$F^t$-semipositivity and Chern characters}\label{S5}
In this section, we show some inequalities of Chern characters for $F^t$-semipositive sheaves. We let $X$ be a smooth projective scheme over a field $k$ with $\dim X=d$ in this section. We freely use the notations in Section \ref{S2} and \ref{S3}.

\begin{theorem}\label{Chern1}
Assume that $\chr(k)=0$. Let $\mathcal{E}$ be an $F^t$-semipositive sheaf on $X$, then $\ch_{d}(\mathcal{E})\geq0$. Moreover, if $\mathcal{E}$ is locally free, then for any closed subscheme $D\subset X$ of dimension $e$ we have $D\cdot\ch_e(\mathcal{E})\geq0$.
\end{theorem}
\begin{proof}
We firstly prove the second statement. We assume that $\mathcal{E}$ is a locally free $F^t$-semipositive sheaf.
Let $i: D\hookrightarrow X$ be the embedding and $\mathcal{I}$ be the ideal sheaf of $D$. Choose a very ample line bundle $\mathcal{O}_X(1)$ and a thickening $(\widetilde{X}, \mathcal{O}_{\widetilde{X}}(1), \widetilde{D}, \widetilde{\mathcal{E}})$ over $\Spec A$.

By Theorem \ref{pullback}, one sees that $i^*\mathcal{E}$ is also $F^t$-semipositive. Then, by the definition of $F^t$-semipositivity one deduces that there exists an integer $m$ such that
$$H^j(D_s, i_s^*(\mathcal{E}^{(p_s^t)}_s)(m))=H^j(D_s, (i^*\mathcal{E})^{(p_s^t)}_s(m))=0$$
for any $j>0$ and almost all $s$. Therefore, from the Riemann-Roch theorem (\cite[Cor. 18.3.1]{Fulton}), if follows that
\begin{eqnarray*}
h^0(D_s, i_s^*(\mathcal{E}^{(p_s^t)}_s)(m)) &=& \chi(D_s, i_s^*(\mathcal{E}^{(p_s^t)}_s)(m)) \\
   &=& \ch(i_s^*(\mathcal{E}^{(p_s^t)}_s)(m))\Td(D_s) \\
   &=& i_s^*\ch(\mathcal{E}^{(p_s^t)}_s(m))\Td(D_s)\\
   &=& D_s\ch_e(\mathcal{E}^{(p_s^t)}_s(m))+\sum_{j=1}^eD_s\ch_{e-j}(\mathcal{E}^{(p_s^t)}_s(m))\Td_j(D_s)\\
   &=& p_s^{te}D_s\ch_e(\mathcal{E}_s)+O(p_s^{t(e-1)}).
\end{eqnarray*}
This implies that $$p_s^{te}D_s\ch_e(\mathcal{E}_s)+O(p_s^{t(e-1)})\geq0$$ for almost all $s$. Taking $p_s\rightarrow\infty$, one concludes that $D\ch_e(\mathcal{E})=D_s\ch_e(\mathcal{E}_s)\geq0$.

One can prove the first statement without the locally free assumption by the same way.
\end{proof}

Let $H$ be a fixed ample divisor on $X$. We define the slope $\mu_H$ of a
coherent sheaf $\mathcal{E}\in \Coh(X)$ by
\begin{eqnarray*}
\mu_H(\mathcal{E})= \left\{
\begin{array}{lcl}
+\infty,  & &\mbox{if}~\rank(\mathcal{E})=0,\\
&&\\
\frac{H^{d-1}\ch_1(\mathcal{E})}{\rank(\mathcal{E})}, & &\mbox{otherwise}.
\end{array}\right.
\end{eqnarray*}

\begin{definition}
A coherent sheaf $\mathcal{E}$ on $X$ is $\mu_{H}$-(semi)stable if, for all non-zero subsheaves
$\mathcal{F}\hookrightarrow \mathcal{E}$, we have
$$\mu_{H}(\mathcal{F})<(\leq)\mu_{H}(\mathcal{E}/\mathcal{F}).$$
\end{definition}
The Harder-Narasimhan filtrations (HN-filtrations, for short)
with respect to $\mu_H$-stability exist in $\Coh(X)$: given a
non-zero sheaf $\mathcal{E}\in\Coh(X)$, there is a filtration
$$0=\mathcal{E}_0\subset \mathcal{E}_1\subset\cdots\subset \mathcal{E}_m=\mathcal{E}$$
such that: $\mathcal{G}_i:=\mathcal{E}_i/\mathcal{E}_{i-1}$ is $\mu_H$-semistable, and
$\mu_H(\mathcal{G}_1)>\cdots>\mu_H(\mathcal{G}_m)$. We set $\mu^+_H(\mathcal{E}):=\mu_H(\mathcal{G}_1)$ and $\mu^-_H(\mathcal{E}):=\mu_H(\mathcal{G}_m)$.

We will use the following estimation theorem of Langer.
\begin{theorem}\label{Langer1}
Assume that $H$ is very ample. Let $\mathcal{E}$ be a rank $r$ torsion free sheaf on $X$. Then
\begin{eqnarray*}
h^0(X, \mathcal{E})\leq\left\{
\begin{array}{lcl}
rH^d\binom{\frac{\mu_H^+(\mathcal{E})}{H^d}+d+f(r)}{d},  & &\mbox{if}~\mu_H^+(\mathcal{E})\geq0,\\
&&\\
0, & &\mbox{otherwise},
\end{array}\right.
\end{eqnarray*}
where $f(r)=-1+\sum_{i=1}^{r}\frac{1}{i}$.
\end{theorem}
\begin{proof}
The result is showed in \cite[Theorem 3.3]{Langer2} when $k$ is algebraically closed. Since the dimension of the space of global sections and the Harder-Narasimhan filtration are stable under base field extension, this result holds for any field.
\end{proof}

By the same reason, another theorem of Langer on the semistability of Frobenius pull backs holds also for an arbitrary field of positive characteristic:

\begin{theorem}\label{Langer2}
Assume that $\chr(k)=p>0$. Let $B$ be a nef divisor on $X$ such that $T_X(B)$ is globally generated. Let $\mathcal{E}$ be a torsion free sheaf on $X$ of rank $r$. Then for any $n\geq0$, we have
$$\max\Big\{\frac{\mu_H^+((F_X^n)^*\mathcal{E})}{p^n}-\mu_H^+(\mathcal{E}), \mu_H^-(\mathcal{E})-\frac{\mu_H^-((F_X^n)^*\mathcal{E})}{p^n}\Big\}\leq\frac{r-1}{p-1}BH^{d-1}.$$
\end{theorem}
\begin{proof}
See \cite[Corollary 2.5]{Langer1}.
\end{proof}

Using the same method in \cite{Sun}, we can also conclude with Bogomolov type inequalities which give upper bounds of Chern characters of sheaves.
\begin{theorem}\label{Bog}
Assume that $\chr(k)=0$. Let $\mathcal{E}$ be a rank $r$ torsion free sheaf on $X$ with $\phi^t_{sp}(\mathcal{E})\leq1$. Then for any $1\leq e\leq d$, we have
$$H^{d-e}\ch_e(\mathcal{E})\leq\max\left\{\frac{r(\mu_H^+(\mathcal{E}))^e}{e!(H^d)^{e-1}}, 0\right\}.$$
\end{theorem}
\begin{proof}
Since the Harder-Narasimhan filtration and $F^t$-semipositivity are stable under base field extension, we may assume that $k$ is algebraically closed. Noting that the conclusion is invariant if we replace $H$ by $nH$ for any $n\geq1$, thus we can assume that $H$ is sufficiently ample so that one can take $$H_1, \cdots H_{d-e}\in |H|$$ satisfying that $X_e:=H_1\cap\cdots\cap H_{d-e}$ is a smooth $e$-dimensional subvariety. By the restriction
theorem (e.g. \cite[Theorem 5.2]{Langer1}) and Proposition \ref{restriction}, one can also assume that $\mathcal{E}|_{X_e}$ is torsion free, $\mu_H^+(\mathcal{E}|_{X_e})=\mu_H^+(\mathcal{E})$ and $\phi^t_{sp}(\mathcal{E}|_{X_e})\leq1$.

Without loss of generality, we may assume that $e=d$, and thus $X_e=X$. Let $(\widetilde{X}\rightarrow\Spec A, \widetilde{\mathcal{E}}, \widetilde{H})$ be a thickening of $(X, \mathcal{E}, H)$. By the openness of semistability, one can shrink $\Spec A$ so that $\mu_{H_s}^+(\mathcal{E}_s)=\mu_H^+(\mathcal{E})$ and $H_s$ is very ample for any $s\in\Spec A$. By the definition of $\phi^t_{sp}(\mathcal{E})$, there exists an integer $m$ such that $H^i(X_s, \mathcal{E}_s^{(p_s^t)}(mH_s))=0$ for $i\geq2$ and almost all $s$. Hence Riemann-Roch theorem gives
\begin{eqnarray}\label{5.1}
\nonumber h^0(X_s, \mathcal{E}_s^{(p_s^t)}(mH_s)) &\geq& \chi(X_s, \mathcal{E}_s^{(p_s^t)}(mH_s))  \\
 \nonumber  &=& \ch(\mathcal{E}_s^{(p_s^t)}(mH_s))\Td(X_s)\\
 \nonumber  &=& \ch_d(\mathcal{E}_s^{(p_s^t)}(mH_s))+\sum_{j=1}^d\ch_{d-j}(\mathcal{E}_s^{(p_s^t)}(mH_s))\Td_j(X_s)\\
   &=& p_s^{td}\ch_d(\mathcal{E}_s)+O(p_s^{t(d-1)}).
\end{eqnarray}
Let $B$ be a nef divisor on $X$ such that $T_X(B)$ is globally generated. Then $T_{X_s}(B_s)$ is aslo globally generated for a general $s\in \Spec A$. From Theorem \ref{Langer1} and \ref{Langer2} it follows that
\begin{eqnarray*}
&&h^0(X_s, \mathcal{E}_s^{(p_s^t)}(mH_s))\\
&\leq&\max\left\{rH_s^d\binom{\mu_{H_s}^+(\mathcal{E}_s^{(p_s^t)})/H_s^d+m+d+f(r)}{d},0\right\}\\
&\leq&\max\left\{rH_s^d\binom{p_s^t\Big(\mu_{H_s}^+(\mathcal{E}_s)+\frac{r-1}{p_s-1}B_sH_s^{d-1}\Big)/H_s^d+m+d+f(r)}{d},0\right\}\\
&\leq&\max\left\{\frac{rp_s^{td}}{d!(H_s^d)^{d-1}}\Big(\mu_{H_s}^+(\mathcal{E}_s)+\frac{r-1}{p_s-1}B_sH_s^{d-1}\Big)^d+O(p_s^{t(d-1)}),0\right\}.
\end{eqnarray*}
This and (\ref{5.1}) imply that
$$\ch_d (\mathcal{E}_s)p_s^{td}\leq\max\left\{\frac{rp_s^{td}}{d!(H_s^d)^{d-1}}\Big(\mu_{H_s}^+(\mathcal{E}_s)+\frac{r-1}{p_s-1}B_sH_s^{d-1}\Big)^d,
0\right\}+O(p_s^{t(d-1)}).$$
The desired inequality follows as $p_s$ approaches infinity.
\end{proof}

\section{Positivity of Hodge bundles and applications}\label{S6}
The aim of this section is to prove the main results and their applications. We let $k$ be a perfect field throughout the section.

\subsection{Decomposition of de Rham complexes}
Before stating the corresponding results let us recall the definition of a semistable reduction (see \cite[Definition 1.1]{Ill}).
\begin{definition}
Let $S$ be a scheme, $X$ and $Y$ smooth $S$-schemes, and $f: X\rightarrow Y$ an $S$-morphism. Let $E\subset Y$ be a divisor relatively simple normal crossing over $S$, and $E_X=f^{-1}(E)=X\times_YE$. One says that $f: X\rightarrow Y$ is $E$-semistable (or semistable for brevity) if, locally for the \'etale topology, $f$ is the product of $S$-morphisms of one of the following types:
\begin{enumerate}
  \item $pr_1: \mathbb{A}_S^n\rightarrow\mathbb{A}_S^1$, $E=\emptyset$;
  \item $h: \mathbb{A}_S^n\rightarrow\mathbb{A}_S^1$, $h^*y=x_1\cdots x_n$, where $\mathbb{A}_S^n=S[x_1,\ldots, x_n]$, $\mathbb{A}_S^1=S[y]$ and $E=(y)$.
\end{enumerate}
\end{definition}

From the definition, one sees that $f$ is $E$-semistable if $f$ is smooth. For $E=\emptyset$, $f$ is $E$-semistable if and
only if $f$ is smooth.
\begin{remark}\label{E-semistable}
Here is a list of basic properties of semistable morphisms in \cite[\S1.2]{Ill}.
\begin{enumerate}
   \item If $f$ is $E$-semistable, and if $(f^{\prime}, E^{\prime})$ is deduced from $(f, E)$ by a base change
$S' \rightarrow S$, then $f^{\prime}$ is $E^{\prime}$-semistable.

  \item If $f$ is $E$-semistable, and $g: Y^{\prime} \rightarrow Y$ is smooth, then $f^{\prime}$ is $E^{\prime}$-semistable, where
$f^{\prime}=f \times_Y Y^{\prime}, E^{\prime}= g^{-1}(E)$.

\item Let $(X_i)_{i\in I}$ (resp. $(Y_i)_{i\in I}$) be a finite non-empty family of smooth $S$-schemes, $f_i: X_i \rightarrow Y_i$ an $S$-morphism, and $E_i \subset Y_i$ a relative divisor with normal crossings. Let $X = \Pi X_i$, $Y = \Pi Y_i$ and $E = \sum pr_i^{-1}(E_i)$. If $f_i$ is $E_i$-semistable for every $i$, then $\Pi f_i$ is $E$-semistable.
\end{enumerate}
\end{remark}

\begin{definition}\label{def6.3}
Let $f: X\rightarrow Y$ be an $E$-semistable $S$-morphism. Let $\Omega^{\bullet}_{X/S}(\log E_X)$ (resp. $\Omega^{\bullet}_{Y/S}(\log E))$ be the de Rham complex of $X$ (resp. $Y$) over $S$ with logarithmic poles along $E_X$ (resp. $E$). We define the coherent sheaf $\Omega^{1}_{X/Y}(\log E_X/E)$ to be the cokernel of the natural injection $f^*\Omega_{Y/S}^1(\log E)\subset \Omega_{X/S}^1(\log E_X)$. Setting $\Omega^{i}_{X/Y}(\log E_X/E)=\wedge^i\Omega^{1}_{X/Y}(\log E_X/E)$ and defining the differential by passage to the quotient of that of the complex $\Omega^{\bullet}_{X/S}(\log E_X)$, we obtain the relative de Rham complex $\Omega^{\bullet}_{X/Y}(\log E_X/E)$ with relative logarithmic poles along $E_X$ over $E$.
\end{definition}

It is easy to see that if $f$ is smooth, then $\Omega^{\bullet}_{X/Y}(\log E_X/E)=\Omega^{\bullet}_{X/Y}$.

\begin{lemma}\label{seq}
With notation as in Definition \ref{def6.3}, there is an exact sequence of locally free sheaves of finite type:
$$0\rightarrow f^*\Omega_{Y/S}^1(\log E)\rightarrow\Omega_{X/S}^1(\log E_X)\rightarrow\Omega^{1}_{X/Y}(\log E_X/E)\rightarrow0.$$
\end{lemma}
\begin{proof}
See \cite[\S 1.3]{Ill}.
\end{proof}
Taking the top wedge product in the exact sequence above gives rise to an isomorphism $$\Omega^{\dim f}_{X/Y}(\log E_X/E)\cong\omega_{X/Y}.$$

Assume that $\chr(k)=p>0$, and Let $f: X\rightarrow Y$ be an $E$-semistable $k$-morphism. The relative Frobenius morphism $F$ of $X$ over $Y$ fits into the commutative diagram
\begin{eqnarray*}
\xymatrix{
 X \ar@/_/[ddr]_f \ar[dr]^{F} \ar@/^/[drr]^{F_X} \\
  & X^{\prime} \ar[d]^{f^{\prime}} \ar[r]_g  & X \ar[d]_f            \\
  & Y \ar[r]^{F_Y} & Y.                       }
\end{eqnarray*}
Set $\Omega^{i}_{X^{\prime}/Y}(\log E^{\prime}_{X}/E)=g^*\Omega^{i}_{X/Y}(\log E_{X}/E)$. The Cartier isomorphism gives:
\begin{proposition}\label{Cartier}
For any $i\geq0$, we have an isomorphism $$C^{-1}:\Omega^{i}_{X^{\prime}/Y}(\log E^{\prime}_{X}/E)\rightarrow\mathcal{H}^i(F_*\Omega^{\bullet}_{X/Y}(\log E_X/E)).$$
\end{proposition}
\begin{proof}
See \cite[Proposition 1.5]{Ill}.
\end{proof}

The following decomposition is proved by Illusie.

\begin{theorem}\label{Illusie-de}
Suppose that $\chr(k)=p>0$. Let $f: (X, E_X)\rightarrow (Y, E)$ be an $E$-semistable $k$-morphism, and $\widehat{f}: (\widehat{X}, \widehat{E}_{\widehat{X}})\rightarrow (\widehat{Y}, \widehat{E})$ a semistable lifting of it over $W_2(k)$. If there exists a lifting $\widehat{F}_Y: \widehat{Y}\rightarrow \widehat{Y}$ of $F_Y$ such that $\widehat{F}_Y^{-1}(\widehat{E})=p\widehat{E}$, then we have a canonical isomorphism in $\D(X^{\prime})$:
$$\phi_{(\widehat{f}, \widehat{F}_Y)}: \bigoplus_{i<p}\Omega^{i}_{X^{\prime}/Y}(\log E^{\prime}_{X}/E)[-i]
\xrightarrow{\sim}\tau_{<p}F_{*}\Omega^{\bullet}_{X/Y}(\log E_X/E),$$ which induces the isomorphism $C^{-1}$ on $\mathcal{H}^i$.
\end{theorem}
\begin{proof}
See \cite[Theorem 2.2]{Ill}.
\end{proof}

\begin{corollary}\label{Illusie-sp}
Let $\widehat{f}: (\widehat{X}, \widehat{E}_{\widehat{X}})\rightarrow (\widehat{Y}, \widehat{E})$ be as in the above theorem. Assume moreover that $f$ is proper. Then:
\begin{enumerate}
  \item For $i+j<p$, the $\mathcal{O}_Y$-modules $R^jf_*\Omega^{i}_{X/Y}(\log E_{X}/E)$ are locally free of finite type, with formation compatible with any base change $Z\rightarrow Y$.
  \item The Hodge spectral sequence $$E_1^{ij}=R^jf_*\Omega^{i}_{X/Y}(\log E_{X}/E)\Rightarrow R^{i+j}f_*\Omega^{\bullet}_{X/Y}(\log E_X/E)$$ satisfies $E_1^{ij}=E_{\infty}^{ij}$ for $i+j<p$.
  \item If $f$ has relative dimension $\leq p$, then (1) and (2) hold for all $(i,j)$.
  \item If $f$ has relative dimension $\leq p$, then the conjugate spectral sequence
  $$_cE_2^{ij}=R^if_*\mathcal{H}^j(\Omega^{\bullet}_{X/Y}(\log E_X/E))\Rightarrow R^{i+j}f_*\Omega^{\bullet}_{X/Y}(\log E_X/E)$$ degenerates in $E_2$.
\end{enumerate}
\end{corollary}
\begin{proof}
The first three conclusions are just \cite[Corollary. 2.4]{Ill}. The last conclusion is showed in \cite[\S 3.2]{Ill}.
\end{proof}

\begin{corollary}\label{relative-vanishing}
Under the assumptions of Corollary \ref{Illusie-sp}, suppose further that $f$ has relative dimension $d\leq p$. Let $\mathcal{L}$ be an $f$-ample line bundle on $X$. Then $$R^jf_*(\mathcal{L}\otimes\Omega^{i}_{X/Y}(\log E_X/E))=0 ~for~ i+j>d.$$
\end{corollary}
\begin{proof}
See \cite[Corollary 2.8]{Ill}.
\end{proof}

We now recall de Rham complexes with coefficients in the Gauss-Manin systems (see \cite[\S 3]{Ill} and \cite[\S 5]{Xie}).
Let $f: X\rightarrow Y$ be an $E$-semistable $k$-morphism, and let $!$ stands for $\dag$ or nothing. We define the graded $\mathcal{O}_Y$-modules
\begin{eqnarray*}
  \mathbb{H}&=&\bigoplus_{i}R^if_*\Omega^{\bullet}_{X/Y}(\log E_X/E);\\
  \mathbb{H}^\dag &=& \bigoplus_{i}R^if_*\left(\Omega^{\bullet}_{X/Y}(\log E_X/E)\otimes\mathcal{O}_X(-E_X)\right).
\end{eqnarray*}
Then there exists the Gauss-Manin connection
$$\nabla: \mathbb{H}^{!}\rightarrow \Omega^1_{Y}(\log E)\otimes\mathbb{H}^{!}$$ which leads to the following complex
$$\Omega^{\bullet}_{Y/S}(\log E)(\mathbb{H}^!)=(\mathbb{H}^!\xrightarrow{\nabla}\Omega^1_{Y}(\log E)\otimes\mathbb{H}^!
\xrightarrow{\nabla}\cdots\xrightarrow{\nabla}\Omega^i_{Y}(\log E)\otimes\mathbb{H}^!\xrightarrow{\nabla}\cdots ).$$
The Hodge filtration on $\mathbb{H}^!$ can be extended to the filtration of $\Omega^{\bullet}_{Y}(\log E)(\mathbb{H}^!)$:
\begin{eqnarray*}
 \Fil^i\Omega^{\bullet}_{Y}(\log E)(\mathbb{H}^!)=(\Fil^i\mathbb{H}^!\xrightarrow{\nabla}\Omega^1_{Y}(\log E)\otimes\Fil^{i-1}\mathbb{H}^!\xrightarrow{\nabla}\\
\cdots\xrightarrow{\nabla}\Omega^j_{Y}(\log E)\otimes\Fil^{i-j}\mathbb{H}^!\xrightarrow{\nabla}\cdots).
\end{eqnarray*}
The associated graded complex of $\Omega^{\bullet}_{Y}(\log E)(\mathbb{H}^!)$ is called Kodaira-Spencer complex:
\begin{eqnarray*}
\gr^{i}\Omega^{\bullet}_{Y}(\log E)(\mathbb{H}^!)=(\gr^i\mathbb{H}^!\xrightarrow{\nabla}\Omega^1_{Y}(\log E)\otimes\gr^{i-1}\mathbb{H}^!\xrightarrow{\nabla}\\
\cdots\xrightarrow{\nabla}\Omega^j_{Y}(\log E)\otimes\gr^{i-j}\mathbb{H}^!\xrightarrow{\nabla}\cdots).
\end{eqnarray*}
If the Hodge spectral sequence
\begin{eqnarray*}
E_1^{ij}=R^jf_*\Omega^i_{X/Y}(\log E_X/E)\Rightarrow  R^{i+j}f_*\Omega^{\bullet}_{X/Y}(\log E_X/E)
\end{eqnarray*}
degenerates, then we have
\begin{eqnarray*}
\gr^i\mathbb{H}^{!}=\left\{
\begin{array}{lcl}
\bigoplus_{j}\mathcal{O}_Y(-E)\otimes R^{j}f_{*}\Omega^i_{X/Y}(\log E_X/E),  & &\mbox{if}~!=\dag,\\
&&\\
\bigoplus_{j}R^{j}f_{*}\Omega^i_{X/Y}(\log E_X/E), & &\mbox{otherwise}.
\end{array}\right.
\end{eqnarray*}

\begin{theorem}\label{Illusie-Xie}
Suppose that $\chr(k)=p>0$. Let $f: (X, E_X)\rightarrow (Y, E)$ be an $E$-semistable $k$-morphism, and $\widehat{f}: (\widehat{X}, \widehat{E}_{\widehat{X}})\rightarrow (\widehat{Y}, \widehat{E})$ a semistable lifting of it over $W_2(k)$. Assume that $f$ is proper and $\dim X<p$, then we have an isomorphism in $\D(Y)$:
$$\phi: \bigoplus_i\gr^{i}\Omega^{\bullet}_{Y}(\log E)(\mathbb{H}^!)\xrightarrow{\sim}F_{Y,*}\Omega^{\bullet}_{Y}(\log E)(\mathbb{H}^!).$$
\end{theorem}
\begin{proof}
The case that $!$ is nothing is prove by Illusie \cite[Theorem 4.7]{Ill}. It is generalized by Xie \cite[Theorem 5.9]{Xie} to the case that $!=\dag$.
\end{proof}

\subsection{Positivity of Hodge bundles}

We prove a more general version of Theorem \ref{main} and Corollary \ref{main-cor} by using the decomposition of de Rham complexes.
\begin{theorem}\label{direct image}
Suppose that $\chr(k)=0$. Let $f: X\rightarrow Y$ be an $E$-semistable $k$-morphism between smooth projective schemes, and $\mathcal{L}$ a ample line bundle on $X$. Then
\begin{enumerate}
  \item $R^if_*\omega_{X/Y}$ is $F^1GG$-semipositive for any $i\geq0$;
  \item $f_*(\omega_{X/Y}\otimes\mathcal{L})$ is $F^1GG$-ample.
\end{enumerate}
In particular, $R^if_*\omega_{X/Y}$ is nef and $f_*(\omega_{X/Y}\otimes\mathcal{L})$ is ample or zero.
\end{theorem}
\begin{proof}
Set $\dim f=d$. Take a very ample line bundle $\mathcal{O}_Y(1)$ on $Y$, and choose a thickening $$(\widetilde{f}: \widetilde{X}\rightarrow \widetilde{Y}, \widetilde{\mathcal{L}}, \widetilde{E}, \mathcal{O}_{\widetilde{Y}}(1))$$ over $\Spec A$. For a general $s\in\Spec A$, we have the commutative diagram
\begin{eqnarray}\label{diag6.1}
\xymatrix{
 X_s \ar@/_/[ddr]_{f_s} \ar[dr]^{F_s} \ar@/^/[drr]^{F_{X_s}} \\
  & X_s^{\prime} \ar[d]^{f_s^{\prime}} \ar[r]_{g_s}  & X_s \ar[d]_{f_s}            \\
  & Y_s \ar[r]^{F_{Y_s}} & Y_s.                       }
\end{eqnarray}
Since $F_s$ is a finite morphism, by Proposition \ref{Cartier} one sees that
\begin{eqnarray*}
 R^if_{s,*}\mathcal{H}^j(\Omega^{\bullet}_{X_s/Y_s}(\log E_{X_s}/E_s)) &\cong& R^if^{\prime}_{s,*}\mathcal{H}^j(F_{s,*}\Omega^{\bullet}_{X_s/Y_s}(\log E_{X_s}/E_s)) \\
   &\cong&  R^if^{\prime}_{s,*}g_s^*\Omega^{j}_{X_s/Y_s}(\log E_{X_s}/E_s) \\
   &\cong&  F^*_{Y_s}R^if_{s,*}\Omega^{j}_{X_s/Y_s}(\log E_{X_s}/E_s).
\end{eqnarray*}
In particular, we have $$R^if_{s,*}\mathcal{H}^d(\Omega^{\bullet}_{X_s/Y_s}(\log E_{X_s}/E_s))\cong F^*_{Y_s}R^if_{s,*}\omega_{X_s/Y_s} .$$
Hence, from Corollary \ref{Illusie-sp} (4), one obtains a surjective morphism
$$R^{d+i}f_{s,*}\Omega^{\bullet}_{X_s/Y_s}(\log E_{X_s}/E_s)\rightarrow F^*_{Y_s}R^if_{s,*}\omega_{X_s/Y_s}$$ for almost all $s\in\Spec A$.

On the other hand, by Corollary \ref{Illusie-sp} (3) one infers that
\begin{eqnarray*}
&& h^j(Y_s, \mathcal{O}_{Y_s}(m)\otimes R^{d+i}f_{s,*}\Omega^{\bullet}_{X_s/Y_s}(\log E_{X_s}/E_s)) \\
&\leq&\sum_{l=0}^d h^j(Y_s, \mathcal{O}_{Y_s}(m)\otimes R^{l}f_{s,*}\Omega^{d+i-l}_{X_s/Y_s}(\log E_{X_s}/E_s)).
\end{eqnarray*}
Thus by semicontinuity we can find an integer $m_0$ such that $$h^j(Y_s, \mathcal{O}_{Y_s}(m_0-j)\otimes R^{d+i}f_{s,*}\Omega^{\bullet}_{X_s/Y_s}(\log E_{X_s}/E_s))=0$$ for almost all $s$. This implies that $$\mathcal{O}_{Y_s}(m_0)\otimes R^{d+i}f_{s,*}\Omega^{\bullet}_{X_s/Y_s}(\log E_{X_s}/E_s)$$ is globally generated. So is $\mathcal{O}_{Y_s}(m_0)\otimes F^*_{Y_s}R^if_{s,*}\omega_{X_s/Y_s}$. One deduces the first conclusion.

For the second conclusion, we consider the complex $\mathcal{L}^{p_s}_s\otimes\Omega^{\bullet}_{X_s/Y_s}(\log E_{X_s}/E_s)$ for a general closed point $s\in\Spec A$, where $p_s=\chr(A/s)$. One sees that
$$F_{s,*}(\mathcal{L}^{p_s}_s\otimes\Omega^{\bullet}_{X_s/Y_s}(\log E_{X_s}/E_s))\cong g_s^*\mathcal{L}_s\otimes F_{s,*}\Omega^{\bullet}_{X_s/Y_s}(\log E_{X_s}/E_s).$$ From Proposition \ref{Cartier}, it follows that
$$\mathcal{H}^i(F_{s,*}(\mathcal{L}^{p_s}_s\otimes\Omega^{\bullet}_{X_s/Y_s}(\log E_{X_s}/E_s)))
\cong g_s^*(\mathcal{L}_s\otimes\Omega^{i}_{X_s/Y_s}(\log E_{X_s}/E_s)).$$
Let $C_s^{\bullet}:=\mathcal{L}^{p_s}_s\otimes\Omega^{\bullet}_{X_s/Y_s}(\log E_{X_s}/E_s)$.
Truncating the complex $F_{s,*}C_s^{\bullet}$, one obtains the exact triangles
\begin{eqnarray*}
\tau_{\leq d-1}F_{s,*}C^{\bullet}_s\rightarrow &F_{s,*}C_s^{\bullet}&\rightarrow g_s^*(\mathcal{L}_s\otimes\omega_{X_s/Y_s})[-d],\\
\tau_{\leq d-2}F_{s,*}C^{\bullet}_s\rightarrow &\tau_{\leq d-1}F_{s,*}C^{\bullet}_s&\rightarrow g_s^*(\mathcal{L}_s\otimes\Omega^{d-1}_{X_s/Y_s}(\log E_{X_s}/E_s))[-d+1],\\
&\vdots&\\
\tau_{\leq 1}F_{s,*}C^{\bullet}_s\rightarrow &\tau_{\leq 2}F_{s,*}C^{\bullet}_s&\rightarrow g_s^*(\mathcal{L}_s\otimes\Omega^{2}_{X_s/Y_s}(\log E_{X_s}/E_s))[-2],\\
g_s^*\mathcal{L}_s\rightarrow &\tau_{\leq 1}F_{s,*}C^{\bullet}_s&\rightarrow g_s^*(\mathcal{L}_s\otimes\Omega^{1}_{X_s/Y_s}(\log E_{X_s}/E_s))[-1].
\end{eqnarray*}
By Corollary \ref{relative-vanishing}, one infers
\begin{eqnarray*}
&&R^if^{\prime}_{s,*}(g_s^*(\mathcal{L}_s\otimes\Omega^{j}_{X_s/Y_s}(\log E_{X_s}/E_s))[-j]) \\
  &\cong&F^*_{Y_s}R^{i-j}f_{s,*}(\mathcal{L}_s\otimes\Omega^{j}_{X_s/Y_s}(\log E_{X_s}/E_s))\\
  &=&0
\end{eqnarray*}
for $i>d$. Hence applying the functor $f^{\prime}_{s,*}$ to the exact triangles above, we have
$$R^{i}f^{\prime}_{s,*}\tau_{\leq j}F_{s,*}C^{\bullet}_s=0$$ for $i>d$. This yields a surjective morphism
\begin{eqnarray*}
\psi_s: &&R^df_{s,*}(\mathcal{L}^{p_s}_s\otimes\Omega^{\bullet}_{X_s/Y_s}(\log E_{X_s}/E_s))\xrightarrow{\sim} R^df^{\prime}_{s,*}F_{s,*}C^{\bullet}_s\\
&\rightarrow& f^{\prime}_{s,*}g_s^*(\mathcal{L}_s\otimes\omega_{X_s/Y_s})\xrightarrow{\sim} F_{Y_s}^*f_{s,*}(\mathcal{L}_s\otimes\omega_{X_s/Y_s}).
\end{eqnarray*}

Let $l$ be any fixed integer. It turns out that the second conclusion follows if one can show the global generation of $$\mathcal{O}_{Y_s}(l)\otimes R^df_{s,*}(\mathcal{L}^{p_s}_s\otimes\Omega^{\bullet}_{X_s/Y_s}(\log E_{X_s}/E_s))$$ for almost all $s$. To this end, we use the following exact triangles
\begin{eqnarray*}
C_s^{\geq1}\rightarrow &C_s^{\bullet}&\rightarrow \mathcal{L}^{p_s}_s,\\
C_s^{\geq2}\rightarrow &C_s^{\geq1}&\rightarrow \mathcal{L}^{p_s}_s\otimes\Omega^{1}_{X_s/Y_s}(\log E_{X_s}/E_s)[-1],\\
&\vdots&\\
\mathcal{L}^{p_s}_s\otimes\omega_{X_s/Y_s}[-d]\rightarrow &C_s^{\geq d-1}&\rightarrow \mathcal{L}^{p_s}_s\otimes\Omega^{d-1}_{X_s/Y_s}(\log E_{X_s}/E_s)[-d+1].
\end{eqnarray*}
By Serre's vanishing and semicontinuity, one sees that
\begin{eqnarray*}
&&R^if_{s,*}(\mathcal{L}^{p_s}_s\otimes\Omega^{j}_{X_s/Y_s}(\log E_{X_s}/E_s)[-j])\\
&=&R^{i-j}f_{s,*}(\mathcal{L}^{p_s}_s\otimes\Omega^{j}_{X_s/Y_s}(\log E_{X_s}/E_s))\\
&=&0
\end{eqnarray*}
for $i>j$ and almost all $s$. Therefore, applying the functor $f_{s,*}$ to the above exact triangles, one obtains
$$R^df_{s,*}C_s^{\bullet}\cong R^df_{s,*}C_s^{\geq1}\cong\cdots \cong R^df_{s,*}C_s^{\geq d-1}$$ and a surjective map
$$f_{s,*}(\mathcal{L}^{p_s}_s\otimes\omega_{X_s/Y_s})\twoheadrightarrow R^df_{s,*}C_s^{\geq d-1}.$$
Since $R^if_{s,*}(\mathcal{L}^{p_s}_s\otimes\omega_{X_s/Y_s})=0$ for $i>0$, one infers that
\begin{eqnarray*}
&&H^i(Y_s, \mathcal{O}_{Y_s}(l-i)\otimes f_{s,*}(\mathcal{L}^{p_s}_s\otimes\omega_{X_s/Y_s}))\\
&\cong& H^i(X_s, f_s^*\mathcal{O}_{Y_s}(l-i)\otimes \mathcal{L}^{p_s}_s\otimes\omega_{X_s/Y_s})\\
&=&0
\end{eqnarray*}
when $p_s\gg0$. Thus $\mathcal{O}_{Y_s}(l)\otimes f_{s,*}(\mathcal{L}^{p_s}_s\otimes\omega_{X_s/Y_s})$ is globally generated for almost all $s$. So is $\mathcal{O}_{Y_s}(l)\otimes R^df_{s,*}D_s^{\bullet}$. This completes the proof.
\end{proof}

\begin{corollary}\label{direct image1}
Let $f: X\rightarrow Y$ be a surjective morphism between complex smooth projective varieties. Let $\mathcal{L}$ be a semiample line bundle on $X$. Then $(f_*(\omega_{X/Y}\otimes\mathcal{L}))^{**}$ is weakly $F^1GG$-semipositive. In particular, $f_*(\omega_{X/Y}\otimes\mathcal{L})$ is weakly $F^1GG$-semipositive if $f$ is flat.
\end{corollary}
\begin{proof}
We firstly show the conclusion for $\mathcal{L}=\mathcal{O}_X$.

By \cite[Proposition 6.1]{Vie83} (see also \cite{AK, ALT}), one can find an open subvariety $Y_0$ of $Y$ and a non-singular covering $\tau: Y^{\prime}_0\rightarrow Y_0$ such that $\codim(Y-Y_0)\geq2$, $f_*\omega_{X/Y}$ is locally free on $Y_0$ and the induced morphism $f^{\prime}_0: X_0^{\prime}\rightarrow Y^{\prime}_0$ is semistable, where $X_0^{\prime}$ is a desingularization of $X\times_Y Y^{\prime}_0$. We have the commutative diagram
$$\xymatrix{
 X_0^{\prime}  \ar[d]_{f_0^{\prime}} \ar[r]^{\tau^{\prime}}  & X_0 \ar[d]^{f_0}            \\
 Y_0^{\prime} \ar[r]^{\tau} & Y_0.                       }$$
It turns out that the proof of Theorem \ref{direct image} still works for complex smooth quasi-projective varieties. Thus applying Theorem \ref{direct image} to $f^{\prime}_0$, one sees that $f^{\prime}_{0,*}\omega_{X_0^{\prime}/Y_0^{\prime}}$ is $F^1GG$-semipositive. By the base change theorem (see \cite{Vie83} and \cite[(4.10)]{Mori}), one obtains an inclusion $$f^{\prime}_{0,*}\omega_{X_0^{\prime}/Y_0^{\prime}}\subset\tau^*f_{0,*}\omega_{X_0/Y_0}.$$ And it is an equality over an open subvariety of $Y^{\prime}_0$. These imply that $\tau^*f_{0,*}\omega_{X_0/Y_0}$ is weakly $F^1GG$-semipositive. Thus $f_{0,*}\omega_{X_0/Y_0}$ is weakly $F^1GG$-semipositive by Lemma \ref{pullback2}. From \cite[Prop. 1.6]{Hart2}, one sees
$$(f_*\omega_{X/Y})^{**}\cong j_*((f_*\omega_{X/Y})^{**}|_{Y_0})\cong j_*((f_*\omega_{X/Y})|_{Y_0})\cong j_*f_{0,*}\omega_{X_0/Y_0},$$ where
$j: Y_0\hookrightarrow Y$ is the inclusion map. Hence $(f_*\omega_{X/Y})^{**}$ is weakly $F^1GG$-semipositive.

We now prove the conclusion for $(f_*(\omega_{X/Y}\otimes\mathcal{L}))^{**}$. Since $\mathcal{L}$ is semiample, $\mathcal{L}^N$ is generated by global sections for some $N>0$. By Bertini's theorem, one finds a non-singular divisor $D$ with $\mathcal{L}^N=\mathcal{O}_X(D)$. Let $Z$ be the cyclic cover obtained by taking the $N$-th root out of $D$ and let $g:Z \rightarrow Y$ be the induced map. Then one finds that $(g_*\omega_{Z/Y})^{**}$ is weakly $F^1GG$-semipositive and it contains $(f_*(\omega_{X/Y}\otimes\mathcal{L}))^{**}$ as a direct summand. This implies $(f_*(\omega_{X/Y}\otimes\mathcal{L}))^{**}$ is weakly $F^1GG$-semipositive.

When $f$ is flat, by \cite[Corollary 5.26]{Horing}, one sees that $f_*(\omega_{X/Y}\otimes\mathcal{L})$ is reflexive. Hence $$f_*(\omega_{X/Y}\otimes\mathcal{L})\cong (f_*(\omega_{X/Y}\otimes\mathcal{L}))^{**}$$ is weakly $F^1GG$-semipositive.
\end{proof}

We give a more general version of Theorem \ref{main-toric} and Corollary \ref{main-toric-cor}.
\begin{theorem}\label{toric}
Suppose that $\chr(k)=0$. Let $f: X\rightarrow Y$ be a smooth morphism between smooth projective $k$-schemes. Assume further that $Y$ is a toric $k$-scheme.
Then $R^jf_*\Omega^i_{X/Y}$ is $F^1$-semipositive for any $i,j\geq0$.
\end{theorem}
\begin{proof}
Set $\dim f=d$. Take a very ample line bundle $\mathcal{O}_Y(1)$ on $Y$, and choose a thickening $$(\widetilde{f}: \widetilde{X}\rightarrow \widetilde{Y}, \mathcal{O}_{\widetilde{Y}}(1))$$ over $\Spec A$. For a general $s\in\Spec A$, we still have the commutative diagram (\ref{diag6.1}). By \cite[Prop. 4.2.2]{AWZ}, one sees that $Y_s$ is a toric $A/s$-scheme for almost all $s$. Hence there exists a lifting of the Frobenius map $F_{Y_s}$ of $Y_s$ over $\Spec (A/s^2)$ by \cite[3.4]{BTLM}. From Theorem \ref{Illusie-de}, one obtains an isomorphism
$$\bigoplus_{i=0}^dg_s^*\Omega^i_{X_s/Y_s}[-i]\cong F_{s,*}\Omega^{\bullet}_{X_s/Y_s}.$$ This implies
\begin{eqnarray*}
  R^jf_{s,*}\Omega^{\bullet}_{X_s/Y_s}&\cong& R^jf^{\prime}_{s,*}(F_{s,*}\Omega^{\bullet}_{X_s/Y_s})\\
  &\cong& \bigoplus_{i=0}^dF^*_{Y_s} (R^{j-i}f_{s,*} \Omega^i_{X_s/Y_s}).
\end{eqnarray*}
From Corollary \ref{Illusie-sp} (3), one sees that there exists an $m_0$ such that
\begin{eqnarray*}
&&\sum_{i=0}^dh^l(Y_s, \mathcal{O}_{Y_s}(m)\otimes F^*_{Y_s} (R^{j}f_{s,*} \Omega^i_{X_s/Y_s}))\\
&=& h^{l}(Y_s, \mathcal{O}_{Y_s}(m)\otimes R^{i+j}f_{s,*}\Omega^{\bullet}_{X_s/Y_s})\\
&\leq&\sum_{a=0}^d h^{l}(Y_s, \mathcal{O}_{Y_s}(m)\otimes R^{i+j-a}f_{s,*}\Omega^{a}_{X_s/Y_s})\\
&=&0
\end{eqnarray*}
for $l>0$, $m>m_0$ and almost all $s$. Therefore $R^jf_*\Omega^i_{X/Y}$ is $F^1$-semipositive.
\end{proof}

\begin{corollary}
Under the situation of Theorem \ref{toric}, we have $$H^b(Y, \Omega^a_Y\otimes R^jf_*\Omega^i_{X/Y}\otimes \mathcal{L})=0$$ for any $b>0$ and ample line bundle $\mathcal{L}$. And for any closed subscheme $D\subset Y$ of dimension $e$, we have $D\cdot\ch_e(R^jf_*\Omega^i_{X/Y})\geq0$ .
\end{corollary}
\begin{proof}
By Theorem \ref{tensor} and \ref{toric}, one sees $R^jf_{*}\Omega^i_{X/Y}\otimes \mathcal{L}$ is $F^1$-ample. The desired vanishing follows from Theorem \ref{Bott}. The inequalities of Chern characters follows from Theorem \ref{Chern1}.
\end{proof}

We now prove a generalization of Theorem \ref{main1} and Corollary \ref{main1-cor}.
\begin{theorem}\label{direct image2}
Suppose that $\chr(k)=0$. Let $f: X\rightarrow Y$ be an $E$-semistable $k$-morphism between smooth projective schemes. If $\Omega^1_Y$ is $F^1$-semipositive, then both $R^if_*\omega_X$ and $R^if_*\omega_X(E_X)$ are $F^1$-semipositive.
\end{theorem}
\begin{proof}
Set $\dim X=n$ and $\dim Y=e$. Take a very ample line bundle $\mathcal{O}_Y(1)$ on $Y$, and choose a thickening $$(\widetilde{f}: \widetilde{X}\rightarrow \widetilde{Y}, \widetilde{E}, \mathcal{O}_{\widetilde{Y}}(1))$$ over $\Spec A$. Let $s\in\Spec A$ be a general closed point and $p_s=\chr(A/s)$. Then by Theorem \ref{Illusie-Xie} we have an isomorphism
$$\bigoplus_i\gr^i\Omega^{\bullet}_{Y_s}(\log E_s)(\mathbb{H}_s^!)\cong F_{Y_s,*}\Omega^{\bullet}_{Y_s}(\log E_s)(\mathbb{H}_s^!).$$
From the Hodge spectral sequence
\begin{eqnarray*}
&&E_1^{ij}=H^{i+j}(Y_s, \mathcal{O}_{Y_s}(m)\otimes F^*_{Y_s}F_{Y_s,*}\gr^i\Omega^{\bullet}_{Y_s}(\log E_s)(\mathbb{H}_s^!))\\
&\Rightarrow&H^{i+j}(Y_s, \mathcal{O}_{Y_s}(m)\otimes F^*_{Y_s}F_{Y_s,*}\Omega^{\bullet}_{Y_s}(\log E_s)(\mathbb{H}_s^!))
\end{eqnarray*}
and Kodaira-Spencer complex, it follows that
\begin{eqnarray*}
&&h^a(Y_s, \mathcal{O}_{Y_s}(m)\otimes F^*_{Y_s}\gr^n\Omega^{\bullet}_{Y_s}(\log E_s)(\mathbb{H}_s^!))\\
&\leq&\sum_ih^a(Y_s, \mathcal{O}_{Y_s}(m)\otimes F^*_{Y_s}\gr^i\Omega^{\bullet}_{Y_s}(\log E_s)(\mathbb{H}_s^!))\\
&=& h^a(Y_s, \mathcal{O}_{Y_s}(m)\otimes F^*_{Y_s}F_{Y_s,*}\Omega^{\bullet}_{Y_s}(\log E_s)(\mathbb{H}_s^!))\\
&\leq&\sum_{i} h^{a}(Y_s, \mathcal{O}_{Y_s}(m)\otimes F^*_{Y_s}F_{Y_s,*}\gr^i\Omega^{\bullet}_{Y_s}(\log E_s)(\mathbb{H}_s^!))\\
&\leq&\sum_{i, j}h^{a-j}(Y_s, \mathcal{O}_{Y_s}(m)\otimes F^*_{Y_s}F_{Y_s,*}(\Omega^{j}_{Y_s}(\log E_s)\otimes\gr^{i-j}\mathbb{H}_s^!)).\\
\end{eqnarray*}
An easy computation shows
\begin{eqnarray*}
\gr^n\Omega_{Y_s}^{\bullet}(\log E_s)(\mathbb{H}_s^{!})=\left\{
\begin{array}{lcl}
\bigoplus_{i\geq0}R^if_{s,*}\omega_{X_s}[-e],  & &\mbox{if}~!=\dag,\\
&&\\
\bigoplus_{i\geq0}R^if_{s,*}\omega_{X_s}(E_{X_s})[-e], & &\mbox{otherwise}.
\end{array}\right.
\end{eqnarray*}
Thus, when $!$ is nothing, one obtains
\begin{eqnarray}\label{6.2}
  &&h^a( \mathcal{O}_{Y_s}(m)\otimes F^*_{Y_s}R^if_{s,*}\omega_{X_s}(E_{X_s}))\\
\nonumber   &\leq& h^{a+e}( \mathcal{O}_{Y_s}(m)\otimes F^*_{Y_s}\gr^n\Omega^{\bullet}_{Y_s}(\log E_s)(\mathbb{H}_s)) \\
\nonumber   &\leq&\sum_{l, j}h^{a+e-j}( \mathcal{O}_{Y_s}(m)\otimes F^*_{Y_s}F_{Y_s,*}(\Omega^{j}_{Y_s}(\log E_s)\otimes\gr^{l-j}\mathbb{H}_s))\\
\nonumber    &\leq&\sum_{l, j, r}h^{a+e-j}(\mathcal{O}_{Y_s}(m)\otimes F^*_{Y_s}F_{Y_s,*}(\Omega^{j}_{Y_s}(\log E_s)
    \otimes R^rf_{s,*}\Omega^{l-j}_{X_s/Y_s}(\log E_{X_s}/E_s))).
\end{eqnarray}

On the other hand, for any locally free sheaf $\widetilde{\mathcal{W}}$ on $\widetilde{Y}$, Xiaotao Sun gives the following filtration in \cite[Theorem 3.7]{SunX}:
$$0=\mathcal{V}_{e(p_s-1)+1}\subset \mathcal{V}_{e(p_s-1)}\subset\cdots\subset \mathcal{V}_1\subset \mathcal{V}_0=F^*_{Y_s}F_{Y_s, *}\mathcal{W}_s,$$
where $\mathcal{V}_j/\mathcal{V}_{j+1}\cong\mathcal{W}_s\otimes \T^j\Omega^1_{Y_s}$ for $0\leq j\leq e(p_s-1)$. The sheaf $\T^l\Omega^1_{Y_s}$ is suited in the exact sequence
\begin{eqnarray*}
0\rightarrow\Sym^{l-l(p_s)p_s}\Omega^1_{Y_s}\otimes F^*_{Y_s}\Omega^{l(p_s)}_{Y_s}\rightarrow \Sym^{l-(l(p_s)-1)p_s}\Omega^1_{Y_s}\otimes F^*_{Y_s}\Omega^{l(p_s)-1}_{Y_s}\\
\rightarrow \cdots\rightarrow\Sym^{l-p_s}\Omega^1_{Y_s}\otimes F^*_{Y_s}\Omega^{1}_{Y_s}\rightarrow\Sym^l\Omega^1_{Y_s}\rightarrow\T^l\Omega^1_{Y_s}\rightarrow0
\end{eqnarray*}
here $l(p_s)\geq0$ is the integer such that $0\leq l-l(p_s)p_s<p_s$.
This implies that
\begin{eqnarray}\label{6.3}
 \nonumber h^b( \mathcal{O}_{Y_s}(m)\otimes F^*_{Y_s}F_{Y_s, *}\mathcal{W}_s) &\leq&
  \sum_{l} h^{b}( \mathcal{O}_{Y_s}(m)\otimes \mathcal{W}_s\otimes \T^l\Omega^1_{Y_s})\\
   &\leq& \sum_{l,j} h^{b+j}( \mathcal{O}_{Y_s}(m)\otimes \mathcal{W}_s\otimes \Sym^{l-jp_s}\Omega^1_{Y_s}\otimes F^*_{Y_s}\Omega^{j}_{Y_s}).
\end{eqnarray}

\begin{claim}\label{claim6.14}
One can find an integer $m_0$ such that $$h^b(\mathcal{O}_{Y_s}(m)\otimes \mathcal{W}_s\otimes \Sym^{i}\Omega^1_{Y_s}\otimes F^*_{Y_s}\Omega^{j}_{Y_s})=0$$ for any $m>m_0$, $b>0$, $i\geq0$, $j\geq0$ and almost all $s$.
\end{claim}
Combining this, (\ref{6.2}) and (\ref{6.3}), it follows that there is an integer $m_0$ such that $$h^a( \mathcal{O}_{Y_s}(m)\otimes F^*_{Y_s}R^if_{s,*}\omega_{X_s}(E_{X_s}))=0$$ for any $m>m_0$, $a>0$, $i\geq0$ and almost all $s$. Hence $R^if_*\omega_{X}(E_X)$ is $F^1$-semipositive. Similarly, one shows $R^if_*\omega_{X}$ is $F^1$-semipositive by taking $!=\dag$.

Turning to Claim \ref{claim6.14}, since $\Omega^1_Y$ is $F^1$-semipositive, so are $\Omega^j_Y$ and $\Sym^{i}\Omega^1_Y$ for any $i, j\geq0$.
These imply that $\reg_0(F^*_{Y_s}\Omega^{j}_{Y_s})$ are bounded for all $j\geq0$ and almost all $s$, and $$H^a(Y, \omega_{Y}\otimes\mathcal{O}_{Y}(m)\otimes\Sym^{i}\Omega^1_Y)=0$$ for any $a>0, m>0$ and $i\geq0$ by Theorem \ref{KV}. Thus $\reg_0(\omega_{Y}\otimes\Sym^{i}\Omega^1_Y)\leq e+1$. So is $\reg_0(\omega_{Y_s}\otimes\Sym^{i}\Omega^1_{Y_s})$ by semicontinuity. From Proposition \ref{tensor-CM}, it follows that
\begin{eqnarray*}
&&\reg_0(\mathcal{W}_s\otimes \Sym^{i}\Omega^1_{Y_s}\otimes F^*_{Y_s}\Omega^{j}_{Y_s})\\
&\leq&\reg_0(\mathcal{W}_s\otimes\omega^{-1}_{Y_s})+\reg_0(\omega_{Y_s}\otimes\Sym^{i}\Omega^1_{Y_s})+\reg_0(F^*_{Y_s}\Omega^{j}_{Y_s})+2(e-1)(R-1)\\
&\leq&\reg_0(\mathcal{W}\otimes\omega^{-1}_{Y})+e+1+\reg_0(F^*_{Y_s}\Omega^{j}_{Y_s})+2(e-1)(R-1),
\end{eqnarray*}
where $R=\max\{1, \reg_0(\mathcal{O}_Y)\}$.
Hence $\reg_0(\mathcal{W}_s\otimes \Sym^{i}\Omega^1_{Y_s}\otimes F^*_{Y_s}\Omega^{j}_{Y_s})$ are bounded for all $i\geq0$, $j\geq0$ and almost all $s$. This completes the proof of Claim \ref{claim6.14}.
\end{proof}

\begin{corollary}
Under the situation of the theorem above, we have
\begin{eqnarray*}
&& H^b(X, \Omega_Y^a(\log E)(\lceil A\rceil)\otimes R^if_*\omega_X)  \\
  &=& H^b(X, \Omega_Y^a(\log E)(\lceil A\rceil-E)\otimes R^if_*\omega_X)\\
  &=&0
\end{eqnarray*}
for $a+b>\dim Y$, where $A$ is an ample $\mathbb{Q}$-divisor on $Y$ such that $\lceil A\rceil-A\leq E$. And for any closed subscheme $D\subset Y$ of dimension $e$, we have $$D\cdot\ch_e(R^if_*\omega_X)\geq0,~ D\cdot\ch_e(\mathcal{O}_Y(E)\otimes R^if_*\omega_X)\geq0.$$
\end{corollary}
\begin{proof}
The vanishing follows from Theorem \ref{KV}. The positivity of Chern characters follows from Theorem \ref{Chern1}.
\end{proof}

\begin{example}\label{eg6.17}
Here are some examples of varieties with $F^1$-semipositive cotangent bundles. Thus Theorem \ref{direct image2} holds when $Y$ is one of these examples.
\begin{enumerate}
  \item Let $A$ be a complex abelian variety of dimension $n$. Then $\Omega^1_A=\mathcal{O}^{\oplus n}_A$ is $F^1$-semipositive.
  \item Let $C_1, \cdots, C_n$ be complex smooth projective curves with positive genus and $X=C_1\times\cdots\times C_n$. Then $\Omega^1_{C_i}=\omega_{C_i}$ is $F^1$-semipositive. So is $$\Omega^1_X\cong\bigoplus_i \pr_i^*\omega_{C_i},$$ where $\pr_i$ is the projection.
  \item Let $f: X\rightarrow Y$ be a nonconstant family of smooth curves. By \cite[Theorem 6.33]{HM}, one sees that $\omega_{X/Y}$ is ample. In particular, it is $F^1$-semipositive. If $\Omega^1_Y$ is $F^1$-semipositive, from Lemma \ref{seq}, it follows that $\Omega^1_X$ is also $F^1$-semipositive.
\end{enumerate}
\end{example}

\begin{remark}\label{Mumford}
Example \ref{eg6.17} (3) shows that $f_*\omega_{X/Y}$ may not be $F^1$-semipositive under the situation of Theorem \ref{direct image}. In fact, let $f: X\rightarrow Y$ be a nonconstant family of smooth curves and let $\dim Y=e$. By Mumford's formula \cite[Chapter XVII, (5.7)]{ACG}, one sees that
\begin{eqnarray*}
\ch_e(f_*\omega_{X/Y})= \left\{
\begin{array}{lcl}
0,  & &\mbox{if}~ 2|e,\\
&&\\
\frac{2(-1)^{\frac{e-1}{2}}\zeta(e+1)}{(2\pi)^{e+1}}K^{e+1}_{X/Y}, & &\mbox{otherwise},
\end{array}\right.
\end{eqnarray*}
where $\zeta$ is the Riemann zeta function. Since $K_{X/Y}$ is ample, one finds $K^{e+1}_{X/Y}>0$. Hence $\ch_e(f_*\omega_{X/Y})<0$ when $\frac{e-1}{2}$ is odd. This implies that $f_*\omega_{X/Y}$ is not $F^1$-semipositive by Theorem \ref{Chern1}.
\end{remark}

\subsection{Applications}
Our main results can be used to study the images of Fano varieties under semistable morphisms and slope inequality.
\begin{theorem}\label{Fano}
Let $f: X\rightarrow Y$ be a surjective morphism between
complex smooth projective varieties. If $f$ is semistable and $-K_X$ is ample, then $-K_Y$ is also ample. If $-K_X$ is semiample, then $-K_Y$ is pseudoeffective.
\end{theorem}
\begin{proof}
If $f$ is semistable and $-K_X$ is ample, by Theorem \ref{direct image}, one sees $$\omega^{-1}_Y=f_*(\omega_{X/Y}\otimes\omega^{-1}_X)$$ is ample. Similarly, the second conclusion follows from Corollary \ref{direct image1} and Proposition \ref{nef}.
\end{proof}


\begin{theorem}\label{slope}
Let $f: X\rightarrow Y$ be a semistable morphism over $\mathbb{C}$ with $\dim X=n$. Assume the general fiber of $f$ is a smooth curve of genus $g$, $K_X$ is nef and $\Omega^1_Y$ is $F^1$-semipositive, then $$K^n_X\geq 2n!\left(\frac{g-1}{g+n-2}\right)\ch_{n-1}(f_*\omega_X).$$
\end{theorem}
\begin{proof}
Take a very ample divisor $H$ on $Y$ and a thickening $(\widetilde{f}: \widetilde{X}\rightarrow \widetilde{Y}, \widetilde{H})$ over $\Spec A$.
For a general $s\in\Spec A$, we have the commutative diagram
\begin{eqnarray*}
\xymatrix{
X_s^{\prime\prime}\ar[r]^{\sigma_s} \ar[dr]_{f_s^{\prime\prime}} & X_s^{\prime} \ar[d]^{f_s^{\prime}} \ar[r]^{g_s}  & X_s \ar[d]_{f_s}            \\
  & Y_s \ar[r]^{F_{Y_s}} & Y_s,                      }
\end{eqnarray*}
where $X_s^{\prime}=Y_s\times_{F_{Y_s}} X_s$ and $\sigma_s: X_s^{\prime\prime}\rightarrow X_s^{\prime}$ is the normalization of $X_s^{\prime}$. Let $C$ be a general fiber of $f_s^{\prime\prime}$ and $L_{m, s}=\sigma_s^*g_s^*K_{X_s}+m(f_s^{\prime\prime})^*H_s$. One sees that $L_{m,s}C=2g-2$.
Applying \cite[Theorem 1.6]{Z2} for $f_s^{\prime\prime}$ and $L_{m,s}$, we have
$$h^0(X_s^{\prime\prime}, \mathcal{O}(L_{m,s}))\leq\left(\frac{1}{2n!}+\frac{n-1}{n!(2g-2)}\right)L_{m,s}^n+g\sum_{i=1}^{n-1}B_i(X_s^{\prime\prime}, H_s, L_{m,s}),$$
where $$B_i(X_s^{\prime\prime}, H_s, L_{m, s})=2^{n-2}\frac{L_{m,s}^n}{2g-2}\frac{L_{m,s}^{n-i}((f_s^{\prime\prime})^*H_s)^i}{L_{m,s}^{n-i+1}((f_s^{\prime\prime})^*H_s)^{i-1}}.$$
Let $p_s=\chr(A/s)$. Since $F^*_{Y_s}H_s=p_sH_s$, a simple computation shows that
\begin{eqnarray*}
  L_{0,s}^n &=& p_s^{n-1}K^n_{X_s},\\
  L_{0,s}^{n-i}((f_s^{\prime\prime})^*H_s)^i&=&p_s^{n-i-1}K^{n-i}_{X_s}(f^*_sH_s)^i,\\
  L_{m,s}^n&=&p_s^{n-1}K^n_{X_s}+O(p_s^{n-2}),\\
  L_{m,s}^{n-i}((f_s^{\prime\prime})^*H_s)^i&=&p_s^{n-i-1}K^{n-i}_{X_s}(f^*_sH_s)^i+O(p_s^{n-i-2}).
\end{eqnarray*}
Substituting them into the inequality above yields:
\begin{equation}\label{6.4}
h^0(X_s^{\prime\prime}, \mathcal{O}(L_{m,s}))\leq\left(\frac{1}{2n!}+\frac{n-1}{n!(2g-2)}\right)p_s^{n-1}K_{X_s}^n+O(p^{n-2}_s).
\end{equation}

On the other hand, from the projection formula, one infers
\begin{eqnarray}\label{6.5}
 h^0((F^*_{Y_s}f_{s,*}\omega_{X_s})(mH_s))&=&h^0((f^{\prime}_{s,*}g_s^*\omega_{X_s})(mH_s)) \\
 \nonumber &=&h^0(g_s^*\omega_{X_s}(m(f^{\prime}_s)^*H_s))\\
 \nonumber &\leq&h^0(\mathcal{O}(L_{m, s}))
\end{eqnarray}
By Theorem \ref{direct image2}, one finds that $f_*\omega_X$ is $F^1$-semipositivie. Hence one can find an integer $m$ such that
$$H^i(Y_s, (F^*_{Y_s}f_{s,*}\omega_{X_s})(mH_s))=0$$ for $i>0$ and almost all $s$. This and the Riemann-Roch theorem give
\begin{eqnarray}\label{6.6}
  h^0((F^*_{Y_s}f_{s,*}\omega_{X_s})(mH_s)) &=& \chi((F^*_{Y_s}f_{s,*}\omega_{X_s})(mH_s)) \\
\nonumber    &=& \ch((F^*_{Y_s}f_{s,*}\omega_{X_s})(mH_s))\Td(Y_s)\\
 \nonumber   &=& p^{n-1}_s\ch_{n-1}(f_{s,*}\omega_{X_s})+O(p_s^{n-2}).
\end{eqnarray}
Combining (\ref{6.4}), (\ref{6.5}) and (\ref{6.6}), we obtain
$$\ch_{n-1}(f_{s,*}\omega_{X_s})\leq\left(\frac{1}{2n!}+\frac{n-1}{n!(2g-2)}\right)K_{X_s}^n+\frac{N}{p_s},$$
here $N$ is a constant independent of $p_s$. Letting $p_s\rightarrow+\infty$, one gets the desired inequality.
\end{proof}


\bibliographystyle{amsplain}

\end{document}